\documentclass[review,onefignum,onetabnum]{siamart190516}

\usepackage{lipsum}
\usepackage{amsfonts}
\usepackage{graphicx}
\usepackage{epstopdf}
\usepackage{algorithmic}
\ifpdf
  \DeclareGraphicsExtensions{.eps,.pdf,.png,.jpg}
\else
  \DeclareGraphicsExtensions{.eps}
\fi

\newsiamremark{remark}{Remark}
\newsiamremark{hypothesis}{Hypothesis}
\crefname{hypothesis}{Hypothesis}{Hypotheses}
\newsiamthm{claim}{Claim}

\headers{diffusive epidemic system in heterogeneous environment}{P. Rui and Y. Wu}

\title{Global $L^\infty$-bounds and long-time behavior of
a diffusive epidemic system in heterogeneous environment \thanks{Submitted to the editors DATE.
\funding{R. Peng was supported by NSF of China (No. 11671175, 11571200), the Priority Academic Program Development of Jiangsu Higher Education Institutions, Top-notch Academic Programs Project of Jiangsu Higher Education Institutions (No. PPZY2015A013) and Qing Lan Project of Jiangsu Province.}}}

\author{Rui Peng\thanks{School of Mathematics and Statistics, Jiangsu Normal University
Xuzhou, 221116, Jiangsu, China
  (\email{pengrui\_seu@163.com}).}
\and Yixiang Wu\thanks{Department of Mathematics, Vanderbilt University
Nashville, TN 37212, USA 
  (\email{yixiang.wu@vanderbilt.edu}).}
}

\usepackage{amsopn}

\ifpdf
\hypersetup{
  pdftitle={Global $L^\infty$-bounds and long-time behavior of
a diffusive epidemic system in heterogeneous environment},
  pdfauthor={R. Peng and Y. Wu}
}
\fi

\newtheorem{Theorem}{Theorem}[section]
\newtheorem{Proposition}{Proposition}[section]

\newtheorem{Lemma}{Lemma}[section]
\newtheorem{Remark}{Remark}[section]
\def\dis{\displaystyle}
\def\Q{Q_{\tau, T}}

\newcommand\bes{\begin{eqnarray}}
\newcommand\ees{\end{eqnarray}}

\begin{document}

\maketitle

\begin{abstract}
 In this paper, we are concerned with
an epidemic reaction-diffusion system with nonlinear incidence mechanism of the form $S^qI^p\,(p,\,q>0)$. The coefficients of the system are spatially heterogeneous and time dependent (particularly time periodic).
We first establish the $L^\infty$-bounds of the solutions of a class of systems,
which improve some previous results  in \cite{Pierre}. Based on such estimates,
we then study the long-time behavior of the solutions of the system.
Our results reveal the delicate effect of the infection mechanism, transmission rate, recovery rate
and disease-induced mortality rate on the infection dynamics.
Our analysis can be adapted to models with some other types of infection incidence mechanisms.
\end{abstract}



\section{Introduction}
Transmission mechanisms play an essential role in susceptible-infected/host-pathogen/host-vector epidemic models \cite{AM,McCallum}. In the pioneering work of
Kermack and McKendrick \cite{KM1}, the disease transmission was assumed to be governed via the
{\it mass action} mechanism: the number of newly infected individuals per unit area per unit of time are
given by a bilinear incidence function $\beta SI$, where $S,\,I$ are the densities of susceptible and
infected individuals, respectively, and $\beta>0$ is the disease transmission rate.
Nevertheless, as pointed out in several works including \cite{DV1,Het1,McCallum},
such a bilinear incidence function carries some shortcomings
and may require modifications in certain situations. As such, many different nonlinear incidence
functions have been proposed to describe the transmission of infectious diseases. One commonly used
nonlinear incidence function takes the form $\beta S^qI^p$, where $p,\,q>0$ are constants.
Epidemic models with this incidence function have been studied extensively;
one may refer to \cite{Het2,HLV,HV,LJ, LHL,LLI,LR} and the references therein.

On the other hand, it has been recognized that environmental heterogeneity and individual motility are significant factors that
should be taken into consideration when studying the  spread  and control of infectious diseases; one may refer to, for instance, \cite{Du-Peng,Martcheva,Ruan} for relevant discussions. Many reaction-diffusion epidemic
models have been developed to investigate the impact of them on the dynamics of disease transmissions, such as malaria
\cite{LZ1,LZ2}, rabies \cite{Kallen1,Kallen2,Murray1}, dengue fever \cite{Takahashi},
West Nile virus \cite{Lewis,Lin}, hantavirus \cite{Abramson1,Abramson2}, Asian longhorned beetle \cite{GL,GZ}, etc.
These models are derived from the ordinary differential equation (ODE) compartmental epidemic models by introducing random diffusion terms to describe the movement of individuals and the spatiotemporally dependent coefficients to describe the environmental heterogeneity.

Taking into account spatial diffusion, environmental heterogeneity as well as a nonlinear incidence
mechanism, we consider the following reaction-diffusion SI/SIS (S: susceptible, I: infected) epidemic system,
which is a natural extension of the ODE epidemic models proposed by Kermack and McKendrick in \cite{KM1,KM2,KM3,KM4}:
 \begin{equation}
 \left\{ \begin{array}{llll}
{   \partial_t} S-d_S\Delta S=-\beta(x,t)S^qI^p+\gamma(x,t)I,&x\in\Omega,\,t>0, \medskip\\
{   \partial_t} I-d_I\Delta I=\beta(x,t)S^qI^p-[\gamma(x,t)+\mu(x,t)]I,&x\in\Omega,\,t>0,\medskip\\
 {   \partial_\nu} S ={   \partial_\nu} I=0,&x\in\partial\Omega,\,t>0,\medskip\\
S(x,0)=S_0(x),\,I(x,0)= I_0(x),&x\in\Omega.\\
 \end{array}\right.
 \label{density}
 \end{equation}
Here, $S(x, t)$ and $I(x, t)$ are the density of susceptible and infected individuals at position $x$ and time $t$, respectively; the habitat $\Omega\subset \mathbb{R}^n$ is a bounded domain with smooth boundary $\partial \Omega$; $\nu$ is the unit outward normal of $\partial\Omega$ and the homogeneous Neumann boundary condition means that there is no population flux across the boundary;
$d_S$ and $d_I$ are positive constants
measuring the motility of susceptible and infected individuals, respectively; $p,\,q>0$ are constants; $\beta$ is the disease transmission
rate; $\gamma$ is the disease recovery rate and $\mu$ is the disease-induced death rate.  If $\gamma=0$, \eqref{density} is an SI model; if $\gamma\neq 0$, it is an SIS model. 

The global dynamics of \eqref{density} with mass action incidence mechanism (i.e., $p=q=1$) have been investigated by several researchers \cite{DW, Fitz2, LiY,  Li,Webb,LiBo1,WZ}.
In \cite{Webb}, Webb studied the case that $\beta$ and $\mu$ are positive constants and $\gamma=0$, and proved that
$S$ converges to a positive number while $I$ decays to zero; among other things, Li and Yip \cite{LiY} derived the same result using a different approach. In \cite{Fitz2},  Fitzgibbon \emph{et al.}, obtained the same asymptotic result with $\beta$ and $\gamma$ being spatially dependent. When $\mu=0$ and
$\beta,\,\gamma>0$ are spatially dependent, \eqref{density} is an SIS epidemic model which was studied by
\cite{DW,LiBo1,WZ}. In \cite{Li}, Li \emph{et al.} introduced a linear source term to the first equation of \eqref{density} to describe the demographic structure of the population. On the other hand, models related to \eqref{density} have been investigated, for instance,  in \cite{CLL,CL, DW,KMP,  Li,LiPengWang,LPX, Peng, Peng-Liu, Peng-Yi, PZ}.   These works were mostly motivated by \cite{Allen}, where Allen \emph{et al.} considered \eqref{density} with  standard incidence mechanism $SI/(S+I)$ and spatially dependent coefficients.

Model \eqref{density} belongs to a class of reaction-diffusion equations that has been studied extensively.  Adding up the first two equations in \eqref{density} and integrating over $\Omega$, we find
\begin{equation}\label{control-mass-d}
\frac{d}{dt} \int_\Omega (S(x, t)+I(x, t)) dx = -\int_\Omega \mu I(x, t) dx,
\end{equation}
and therefore
\begin{equation}\label{control-mass}
\int_\Omega (S(x, t)+I(x, t)) dx\le  \int_\Omega (S_0(x)+I_0(x)) dx:=N, \ \ t\ge 0,
\end{equation}
which means that the total population is bounded by $N$.
Models with such a property are called reaction-diffusion systems with {\it control of mass}, which  are of the form:
\begin{equation}\label{uv}
\left\{
\begin{array}{lll}
{   \partial_t} u-d_1\Delta u=f(u, v), \ \ &x\in{\color{red} \Omega}, t>0,  \medskip\\
{   \partial_t} v-d_2\Delta v=g(u, v), \ \ &x\in{\color{red}\Omega}, t>0,  \medskip\\
{   \partial_\nu} u= {   \partial_\nu} v=0, \ \  &x\in\partial\Omega, t>0, \medskip\\
u(x, 0)=u_0(x)\ge 0, \ \ v(x, 0)=v_0(x)\ge 0, \ \ &x\in\Omega,
\end{array}
\right.
\end{equation}
where $f, g\in C^1(\mathbb{R}_+^2)$ satisfies
\begin{enumerate}
\item[{\rm(P)}]\ \ \ {   $f(0, v)\ge 0 \ \text{ and } \  g(u, 0)\ge 0,\ \ \forall  u, \,v\ge 0$},
\end{enumerate}
and
\begin{enumerate}
\item[{\rm(M)}]\ \ \
{   $f(u, v)+g(u, v)\le C(1+u+v),\ \ \forall u, \,v\ge 0$},
\end{enumerate}
where $C$ is some nonnegative number.

In \cite{Pierre}, (P) is called as the {\it quasi-positive} condition and (M) is the {\it mass-control} condition.  When $f=-uv^p$ and $g=uv^p$, Alikakos \cite{Al} proved the global existence and $L^\infty$ boundedness of the solutions using the Alikakos-Moser iteration technique for the case $1\le p <(n+2)/n$. In \cite{Masuda}, Masuda  dropped the assumption on $p$ using a Lyapunov functional method. Moreover, Masuda proved the convergence of the solutions to nonnegative constants.  When $f=-u^qv^p+ u^sv^r$ and $g=u^sv^r-u^qv^p$, the system was studied in \cite{Kouachi, Pierre}. For general functions $f$ and $g$ satisfying (P) and (M), Hollis \emph{et al.} \cite{Hollis} studied the $L^\infty$ boundedness of the solutions, where  a prior $L^\infty$-bound on $u$ was assumed. In \cite{Morgan1, Morgan2}, Morgan studied a more general reaction-diffusion system with control in mass, where the boundedness results were based on a Lyapunov-type condition on the nonlinearities. We refer the interested readers to the survey paper by Pierre \cite{Pierre} for more studies along this direction.

%

A priori $L^\infty$-estimates are crucial in the study of the long-time behavior of the solutions of  \eqref{density}. System \eqref{uv} admits a priori $L^1$-estimates, namely, \eqref{control-mass}.  However, it is a rather challenging problem to bootstrap a priori $L^1$-estimates to $L^\infty$-estimates. Indeed, concrete examples for \eqref{uv} have been found that a priori $L^1$-estimates may lead to blow-up at finite time in the $L^\infty$-norm; see \cite{Pierre,QS} and the references therein.
Thus, in order to bootstrap $L^1$-estimates to $L^\infty$-estimates, one has to impose extra conditions on the reaction terms.
In this paper, we will show that the solutions of \eqref{density} are $L^\infty$-bounded for all $p, q>0$ and \eqref{density} is dissipative (i.e., solutions are ultimately uniformly $L^\infty$-bounded).


With our $L^\infty$-estimates, we can investigate the long-time behavior of the solutions of system \eqref{density}.  In the case  $\mu>0$,  we prove that the $I$-component of the solution decays to zero, while the $S$-component of the solution converges to some nonnegative constant $S^*$. More importantly, we show that $S^*$ is positive if $p\ge 1$ and $S^*=0$ if $0<p<1$. In the case $\mu=0$, we prove that \eqref{density} is uniformly persistent under certain conditions.  We point out that one of the main difficulties of the analysis comes from the fact that $S^qI^p$ is not Lipschitz when $0<p, q<1$; for example, the solutions of \eqref{density} may not induce a semiflow on a complete metric space, and therefore many existing theories on dynamical systems cannot apply directly. We remark that the case $0<p<1$ is important, which is the main consideration in \cite{FCGT, LJ}. 


The paper is organized as follows. In Section 2,  we state the main results; in Section 3, we derive the $L^\infty$-bounds of the solutions of a class of reaction-diffusion systems with control of mass; in Section 4, we prove the positivity, uniqueness and $L^\infty$-boundedness of the solutions of \eqref{density}; in Sections 5 and 6, we investigate the long-time behavior of the solutions; in Section 7, we discuss some possible generalizations and the biological implications of our results; in the appendix, we provide some helpful results on the ODE  epidemic systems, which are usually the guidance for the analysis of the corresponding reaction-diffusion epidemic systems, and we state and generalize some results on dynamical systems defined on incomplete metric spaces,  which are used in the proofs in Section 6.


\section{Preliminaries and main results}

\subsection{Global existence and $L^\infty$-bounds of \eqref{uv}}
We first recall some results about the global existence of the solutions of \eqref{uv}. In the survey paper \cite{Pierre}, Pierre established the following result:

\begin{Proposition}[{\cite[Theorem 3.1]{Pierre}}]   \label{lemma_global}
Suppose that {\rm(P)}-{\rm (M)}, and the following hold:
$$
f(u, v)\le C(1+u+v), \ \
 |g(u, v)|\le C(1+u^a+v^a),
 $$
for all $u\ge U, v\ge 0$, where the constants $U, C\ge 0,\,a>0$.
Then, for any nonnegative initial data $u_0, v_0\in C(\bar\Omega)$,
the classical solution of \eqref{uv} is nonnegative and exists globally.
\end{Proposition}

In \cite{Pierre}, the author applied Proposition \ref{lemma_global} to system \eqref{uv}
with $f=-u^qv^p+ \lambda  u^sv^r$, $g=u^sv^r-u^qv^p$ and $\lambda\in [0, 1]$,
and derived sufficient conditions for the global existence of solutions. More precisely, for any nonnegative initial data $u_0, v_0\in C(\bar\Omega)$, the system has a unique nonnegative global classical solution if
\begin{equation}\label{pqrs}
 p,\, q,\, r,\, s\ge 1
\end{equation}
and one of (i)-(iv) in the following holds:
\begin{enumerate}
\item[{\rm(H1)}]\ {\rm (i)}\ $p>r\geq0$ and $sp-rq\le p-r$;\ \ \ \ {\rm (ii)}\ $p=r\geq0$ and $0\leq s<q$;\\
 {\rm (iii)}\ $s>q\geq0$ and $sp-rq\le s-q$;\ \ \
 {\rm (iv)}\ $s=q\geq0$ and $0\leq p<r$.
\end{enumerate}

{  
\begin{Proposition}[{\cite[Theorem 3.1]{Pierre}}]  \label{lemma_global-a} Assume that \eqref{pqrs} and  one of (i)-(iv) in {\rm(H1)} hold.
Then, for any nonnegative initial data $u_0, v_0\in C(\bar\Omega)$,
system \eqref{uv} with $f=-u^qv^p+ \lambda  u^sv^r$, $g=u^sv^r-u^qv^p$ for $\lambda\in [0, 1]$ admits a unique nonnegative global classical solution.
\end{Proposition}
}

To include \eqref{density} and \eqref{uv} with $f=-u^qv^p+ \lambda u^sv^r, g=u^sv^r-u^qv^p$, 
we consider  the following system: 
 \begin{equation} \label{density1}
 \left\{ \begin{array}{llll}
 \dis {   {\partial_t u}}-d_u\Delta u=-\beta u^qv^p+\gamma u^sv^r,&x\in\Omega,\,t>0, \medskip\\
  \dis {   {\partial_t v}}-d_v\Delta v=\beta u^qv^p-[\gamma +\mu]u^sv^r,&x\in\Omega,\,t>0,\medskip\\
 \dis {   \partial_\nu u}=\dis {   \partial_\nu v}=0,&x\in\partial\Omega,\,t>0,\medskip\\
 \dis u(x,0)=u_0(x),\,v(x,0)= v_0(x),&x\in\Omega,
 \end{array}\right.
 \end{equation}
where $p, q, s, r$ are nonnegative constants and $u_0, v_0\in C(\bar\Omega)$ are nonnegative. {   The coefficients $\beta$, $\gamma$ and $\mu$ are assumed to be nonnegative constants temporarily as they do not play an essential role in establishing the boundedness results here. }

The results in \cite{Pierre} require condition \eqref{pqrs}. In this paper, we will
show that, without condition \eqref{pqrs}, once the solution of system \eqref{density1} exists globally,
it must be uniformly bounded and the system is dissipative provided that either case in {\rm(H1)} holds.
Moreover, we are able to obtain some different parameter ranges which ensure the uniform boundedness of solution;
that is, we will deal with the following parameter ranges:
\begin{enumerate}
\item[{\rm(H2)}]\ {\rm (i)}\ $0\leq p<1$, $0\le s<q$, $ps-qr\ge s-q$;\ \\ {\rm (ii)}\  $0\leq s<1$, $0\leq p<r$, $ps-qr\ge p-r$.
\end{enumerate}

Throughout this paper, $N$ is the initial total mass, i.e. for system \eqref{uv}
$$
N=\int_\Omega (u_0(x)+ v_0(x)) dx,
$$ 
 while $N$ is given by \eqref{control-mass} for system \eqref{density}.

Our $L^\infty$-boundedness results on \eqref{density1} can be stated as follows.

\begin{Theorem}\label{Theorem_boundaa}
Suppose that $\beta>0$, $\gamma, \mu\ge 0$, and one of {\rm (H1)-(i)}, {\rm (H1)-(ii)} or {\rm (H2)-(i)} holds. Let $(u, v)$ be a  nonnegative classical solution of \eqref{density1}. Then there exists $M_\infty>0$ depending on the initial data such that
\begin{equation}\label{Minfdd}
\|u(\cdot, t)\|_{L^\infty(\Omega)},\ \  \|v(\cdot, t)\|_{L^\infty(\Omega)} \le M_\infty, \ \forall t\ge 0.
\end{equation}
Suppose in addition that  $p=r=1$ when  {\rm (H1)-(ii)} holds. Then, there exists $N_\infty$ depending only on $N$ such that, for any global classical solution $(u, v)$, the following hold:
\begin{equation}\label{Ninfdd}
\limsup_{t\rightarrow\infty}\|u(\cdot, t)\|_{L^\infty(\Omega)},\ \ \limsup_{t\rightarrow\infty}\|v(\cdot, t)\|_{L^\infty(\Omega)} \le N_\infty.
\end{equation}
\end{Theorem}

\begin{Remark}
Eq. \eqref{Ninfdd} means that the system is dissipative, which is essential for proving the uniform persistence of \eqref{density} later on. If \eqref{pqrs} is true, then \eqref{Minfdd} assures that a local solution can be uniquely extended to be a global solution. If \eqref{pqrs} is not assumed, the solution may fail to be unique since $u^qv^p$ or $u^sv^r$ may not be Lipschitz.
\end{Remark}

In the case that one of {\rm (H1)-(iii)}, {\rm (H1)-(iv)} or {\rm (H2)-(ii)} holds, we can also establish the uniform boundedness result just by exchanging the roles of $u$ and $v$ in our analysis of Section 3 below. 

\begin{Theorem}\label{Theorem_bound1}
{   Suppose that $\beta\ge 0$, $\gamma+\mu> 0$, and one of {\rm (H1)-(iii)}, {\rm (H1)-(iv)} or {\rm (H2)-(ii)} holds.} Let $(u, v)$ be the a nonnegative  classical solution of \eqref{density1}. Then there exists $M_\infty>0$ depending on the initial data such that
\eqref{Minfdd} holds. Suppose in addition that $q=s=1$ when {\rm (H1)-(iv)} holds. Then, there exists $N_\infty$ depending only on $N$ such that \eqref{Ninfdd} holds for any  global classical solution $(u, v)$.
\end{Theorem}


\subsection{Global existence, positivity, $L^\infty$-bounds and dissipativity of  \eqref{density}}
We now turn to system \eqref{density} and let the coefficients $\beta(x, t), \gamma(x, t), \mu(x, t)$ be dependent on $x$ and $t$.  We impose the following assumptions for \eqref{density}:{  
\begin{enumerate}
\item[{\rm(A1)}] The functions $\beta, \gamma, \mu$ are H\"{o}lder continuous with
$$
0\le \beta(x, t),\ \gamma(x, t),\ \mu(x, t)\le \sigma^0 \ \  \mbox{on} \ \bar\Omega\times[0,\infty),
$$
and, there exists some constant $\sigma_0>0$ such that $\beta(x, t)>\sigma_0$ for all $x\in\bar\Omega$ and $t>0$.
\item[{\rm(A2)}] $S_0$ and $I_0$ are nonnegative continuous functions on $\bar\Omega$. Moreover, 
\begin{enumerate}
\item[(i)] The initial value $I_0\geq,\not\equiv0$ on $\bar\Omega$;
\item[(ii)] If $0<q< 1$, $S_0(x)>0$ for all $x\in\bar\Omega$, and there exists $\sigma_0>0$ such that $\gamma\ge \sigma_0$ for all $x\in\bar\Omega$ and $t\ge 0$;
\item[(iii)] If $0<p<1$, $I_0(x)>0$ for all $x\in\bar\Omega$.
\end{enumerate}
\end{enumerate}
}

%


%
%

 \begin{Remark}
{   Biologically, {\rm (A1)} means that the disease transmission, recovery and mortality rates are bounded and nonnegative.  {\rm (A2)} is imposed to guarantee the global existence and the positivity of the solutions of \eqref{density}}:
 \begin{itemize}
\item If $I_0=0$, then the component $I(x,t)=0$ for all $x\in\bar\Omega,\,t\ge 0$, which is not of our interests.
\item If $\gamma=0$ and $0<q<1$, it is possible for $S(x, t)$ to vanish at some finite time; see Proposition \ref{ODE-result}-(i) in the appendix.
\item When either $0<p<1$ or $0<q<1$, the term $S^qI^p$ is not locally Lipschitz unless $S, I>0$. Therefore, we have to assume the positivity of the initial data to ensure the uniqueness of the local solution. Moreover, the unique extension of the local solution requires the positivity of the solution.
\end{itemize}
 \end{Remark}

We have the following result on the global existence, positivity and $L^\infty$-bounds of the solutions of \eqref{density}.
\begin{Theorem}\label{th-a}
Assume that {\rm (A1)-(A2)} hold. Then system \eqref{density} has a unique global classical solution $(S, I)$ for any $p, q>0$ with $S(x, t), I(x, t)>0$ for all $x\in\bar\Omega$ and $t>0$, and there exists $M_\infty>0$ depending on the initial data such that
\begin{equation}\label{th-a1}
\|S(\cdot, t)\|_{L^\infty(\Omega)}, \ \ \|I(\cdot, t)\|_{L^\infty(\Omega)} \le M_\infty, \ \  t\ge 0.
\end{equation}
Moreover, there exists $N_\infty$ depending only on $N$ such that
\begin{equation}\label{th-a2}
\limsup_{t\rightarrow\infty}\|S(\cdot, t)\|_{L^\infty(\Omega)}, \ \ \limsup_{t\rightarrow\infty}\|I(\cdot, t)\|_{L^\infty(\Omega)} \le N_\infty.
\end{equation}
\end{Theorem}


\subsection{ Long-time behavior of \eqref{density}} In this section, we state our results on the long-time behavior of the solutions of \eqref{density}.

\noindent\textbf{(a) Case $\mu>0$ on $\bar\Omega\times[0,\infty)$}

\vskip5pt
 In this case, we further assume that
\begin{itemize}
\item[{\rm (A3)}] There exists $\sigma_0>0$ such that $\mu(x, t)\ge \sigma_0$ for all $x\in\bar\Omega$ and $t\ge 0$.
 \end{itemize}
{   Biologically, (A3) means that there are losses of infected individuals due to disease-induced mortality.}  Under (A3), our  main result about the dynamical behavior of the solution of \eqref{density} reads as follows.

 \begin{Theorem}\label{result1}
Assume that {\rm (A1)-(A3)} hold. Let $(S,I)$ be the unique solution of \eqref{density}. Then the following assertions hold:
 \begin{itemize}
 \item[\rm{(i)}] If $p\geq1$, we have $(S,I)\to(S_*,0)$ uniformly on $\bar\Omega$ as $t\to\infty$, where
  $S_*$ is a positive constant. Moreover,
  $S_*\in (0,  \sup_{\Omega\times(0,\infty)}{(\gamma+\mu)}/{\beta}]$ if $p=1$.
 \item[\rm{(ii)}] If $0<p<1$, we have
  $(S,I)\to(0,0)$ uniformly on $\bar\Omega$ as $t\to\infty$.
 \end{itemize}
 \end{Theorem}

\noindent\textbf{(b) Case $\mu=0$ on $\bar\Omega\times[0,\infty)$} 
%

\vskip5pt {   In this case, the disease is not fatal.} We further assume that the coefficients are periodic in time, i.e.,

\begin{enumerate}
\item[(A4)] The functions $\beta$ and $\gamma$ are periodic  with the common period $\omega>0$, i.e., $\beta(x, t+\omega)=\beta(x, t)$ and $\gamma(x, t+\omega)=\gamma(x, t)$ for all $x\in\bar\Omega$ and $t\ge 0$.
\end{enumerate}

{   Biologically, (A4) means that the epidemic is seasonal.} We have the following result about the uniform persistence of the solutions and the existence of a positive $\omega$-periodic  solution of \eqref{density}. The definition of \emph{the basic reproduction number}, $\mathcal{R}_0$,  is given by \eqref{Rdef}. 
\begin{Theorem}[Uniform persistence]\label{th-a4}
Let $\mu=0$. Assume that {\rm(A1)-(A2)} and {\rm(A4)} hold, and let $0<p\leq1$. We further assume $\mathcal{R}_0>1$ if $p=1$.
Then there exists $\epsilon_0>0$ depending only on $N$ such that for any solution $(S, I)$ of \eqref{density}, we have
\begin{equation}\label{persist}
\liminf_{t\rightarrow\infty} S(x, t),\ \  \liminf_{t\rightarrow\infty}I(x, t) \ge \epsilon_0,
\end{equation}
uniformly for $x\in\bar\Omega$. Moreover, \eqref{density} has at least one positive $\omega$-periodic solution.
\end{Theorem}

%
%

\begin{Remark}\label{re-p1} We would like to make the following comments for Theorem \ref{th-a4}.
 \begin{itemize}
 \item[\rm{(i)}] When $p=1$, we suspect that if $\mathcal{R}_0>1$, \eqref{density} admits a unique positive equilibrium which is globally attractive, whereas the unique disease-free ($I$-component is zero) equilibrium is globally attractive if $\mathcal{R}_0\leq1$. One may follow the same analysis as in \cite[Theorems 4.1, 4.2]{DW} to prove this for the special cases when either $\beta,\,\gamma$ are positive constants or $d_S=d_I$. 

 \item[\rm{(ii)}] When $p>1$, Proposition \ref{ODE-SIS} of the appendix analyzes the long-time behavior of the ODE version of \eqref{density} which suggests that the dynamics of \eqref{density} depends on the initial data.

 \end{itemize}
\end{Remark}


\section{Proof of Theorem \ref{Theorem_boundaa}}

\subsection{Proof of Theorem \ref{Theorem_boundaa} under condition (H1)-(i) or (H1)-(ii)}

In this subsection, we obtain the $L^\infty$-bounds of the solutions of \eqref{density1}
provided that one of (H1) is fulfilled. In what follows, we focus on  two cases: {\rm (H1)-(i)}:\ $p>r\geq0$ and $sp-rq\le p-r$; {\rm (H1)-(ii)}:\ $p=r\geq0$ and $0\leq s<q$. The other two cases in (H1) can be handled by exchanging the roles of $u$ and $v$.

We begin with the following useful lemma.
\begin{Lemma}\label{lemma_inequality}
Assume that  {\rm (H1)-(i)} holds. Then for any $\epsilon>0$, we have
\begin{equation}\label{pq}
u^{s}v^r\le \epsilon u^{q}v^{p}+A_1 u+A_2,\ \ \forall u\ge 1,\ v\ge 0,
\end{equation}
where $A_1$ and $A_2$ are positive constants depending only on $\epsilon$.
\end{Lemma}
\begin{proof} By the well-known Young's inequality, for any $\epsilon>0$, there exists $C_\epsilon>0$ such that
\begin{eqnarray*}
u^{s}v^r &=& (u^{q\frac{r}{p}}v^r)  u^{s-q\frac{r}{p}} \\
&\le& \epsilon u^{q}v^p +C_\epsilon u^{(s-q\frac{r}{p})\frac{p}{p-r}}\\
&=&  \epsilon u^{q}v^p +C_\epsilon u^{\frac{ps-qr}{p-r}},\ \ \forall u,\, v\ge 0.
\end{eqnarray*}
Thus, the claimed inequality follows from the conditions $ps-qr\le p-r$ and $u\geq1$.
\end{proof}

{   For Lemmas \ref{lemma_s}-\ref{lemma_s1}, Lemma \ref{lemma_dual} and Lemma \ref{lemma_t}, we suppose that   $\beta>0$, $\gamma, \mu\ge 0$, and {\rm (H1)-(i)} holds. Let $(u, v)$ be a nonnegative  classical solution of \eqref{density1}. }

\begin{Lemma}\label{lemma_s} For any nonnegative integer $k$, there exists $M_{2^k}>0$ depending on the initial data such that
\begin{equation}\label{mk}
\|u(\cdot, t)\|_{L^{2^k}(\Omega)}^{2^k} \le M_{2^k},\  \ \forall t\ge 0.
\end{equation}
Moreover, there exists $N_{2^k}>0$ depending only on $N$ such that, for any global classical solution $(u, v)$, the following holds:
\begin{equation}\label{nk}
\limsup_{t\rightarrow\infty} \|u(\cdot, t)\|_{L^{2^k}(\Omega)}^{2^k} \le N_{2^k}.
\end{equation}
\end{Lemma}
\begin{proof} We will use an induction argument to derive \eqref{mk} and \eqref{nk}.
Apparently, \eqref{mk} and \eqref{nk} hold with $M_{1}=N_{1}=N$ when $k=0$.

Let $k\ge 1$. Suppose that there exists $M_{2^{k-1}}>0$ depending on the initial data such that
\begin{equation}\label{mk1}
\|u(\cdot, t)\|_{L^{2^{k-1}}(\Omega)}^{2^{k-1}} \le M_{2^{k-1}}, \ \ \forall t\ge 0,
\end{equation}
and there exists $N_{2^{k-1}}\ge 0$ depending only on $N$ such that
\begin{equation}\label{nk1}
\limsup_{t\rightarrow\infty}\|u(\cdot, t)\|_{L^{2^{k-1}}(\Omega)}^{2^{k-1}} \le N_{2^{k-1}}.
\end{equation}

For notational convenience, let us set
 $$
 \bar u:=(u-1)_+=\max\{u-1, 0\}.
 $$
Multiplying both sides of the first equation of \eqref{density1} by $\bar u^{2^k-1}$ and integrating on $\Omega$, we obtain
\begin{equation}\label{e0}
\frac{1}{2^k}\frac{d}{dt}\int_\Omega \bar u^{2^k} dx \le -\frac{2^k-1}{2^{2k-2}} d_u\int_\Omega |\nabla \bar u^{2^{k-1}}|^2dx + \int_\Omega (-\beta u^{q}v^p+ \gamma u^{s}v^r)\bar u^{2^k-1} dx, \ \ \forall t\ge 0.
\end{equation}
By Lemma \ref{lemma_inequality}, for any $\epsilon>0$, one can find two positive constants $A_1=A_1(\epsilon)$ and $A_2=A_2(\epsilon)$ such that
$$
u^{s}v^r\le \epsilon u^{q} v^p + A_1 u+ A_2,\ \ \forall u\geq1,\,v\geq0.
$$
{   If $\gamma>0$, choosing $\epsilon=\beta/\gamma$} and using the definition of $\bar u$ and H\"{o}lder's inequality, we have
\begin{equation}
\label{e1}
\int_\Omega (-\beta u^{q}v^p+ \gamma u^{s}v^r)\bar u^{2^k-1} dx \le  C_1\int_\Omega \bar u^{2^k}dx+C_2, \ \ \forall t\ge 0.
\end{equation}
where $C_1=2{   \gamma}A_1$, $C_2={   \gamma}(A_1+A_2)|\Omega|$. {   If $\gamma=0$, \eqref{e1} holds with $C_1=C_2=1$. }

We recall the following interpolation inequality: for any $\epsilon_*>0$, there exists $C_{\epsilon_*}>0$ such that
\begin{equation}\label{xii}
\|w\|_{L^2(\Omega)}^2 \le \epsilon_* \|\nabla w\|_{L^2(\Omega)}^2+C_{\epsilon_*}\|w\|_{L^1(\Omega)}^2, \ \quad \forall w\in W^{1, 2}(\Omega).
\end{equation}

Setting $\epsilon_*=\frac{2^k-1}{2C_12^{2k-2}}d_u$ and $w=\bar u^{2^{k-1}}$ in \eqref{xii}, we have
\begin{eqnarray}\label{e2}
\ \ \ \ \ \ \ \ -\frac{2^k-1}{2^{2k-2}} d_u\int_\Omega |\nabla \bar u^{2^{k-1}}|^2dx \le  -2C_1 \int_\Omega \bar u^{2^k} dx + C_3\left(\int_\Omega \bar u^{2^{k-1}}dx\right)^2, \ \ \forall t\ge 0,
\end{eqnarray}
where $C_3=2C_1C_{\epsilon_*}/{\epsilon_*}$. Combining \eqref{e0}-\eqref{e2}, we deduce
\begin{equation}\label{ubar}
\frac{1}{2^k} \frac{d}{dt} \int_\Omega \bar u^{2^k} dx \le -C_1\int_\Omega \bar u^{2^k} dx +C_2+C_3\left(\int_\Omega \bar u^{2^{k-1}}dx\right)^2, \ \ \forall t\ge 0.
\end{equation}

Since $0\le \bar u\le u$, \eqref{mk1}-\eqref{nk1} hold with $u$ replaced by $\bar u$. It then follows from \eqref{ubar} that \eqref{mk}-\eqref{nk} hold  with $u$ replaced by $\bar u$, where
$$
M_{2^k}= \max\left\{\frac{C_2+C_3M_{2^{k-1}}^2}{C_1}+1,   \int_\Omega S_0^{2^k}dx \right\}, \ \ \ N_{2^k}=\frac{C_2+C_3N_{2^{k-1}}^2}{C_1}.
$$
Therefore, the lemma follows from the fact  $u\le \bar u+1$.
\end{proof}

\vskip5pt
In view of Lemma \ref{lemma_s} and the embedding $L^{a_1}(\Omega)\subset L^{a_2}(\Omega)$ for any $a_1>a_2$, we have the following observation.

\begin{Lemma}\label{lemma_s1}
For any constant $a\geq1$, there exists $M_{a}>0$ depending on the initial data such that
\begin{equation*}
\|u(\cdot, t)\|_{L^{a}(\Omega)} \le M_a, \ t\ge 0.
\end{equation*}
Moreover, there exists $N_a>0$ depending only on $N$ such that, for any global classical solution $(u, v)$, the following holds:
\begin{equation*}
\limsup_{t\rightarrow\infty} \|u(\cdot, t)\|_{L^{a}(\Omega)}^{a} \le N_a.
\end{equation*}
\end{Lemma}

Our argument below to prove the uniform bounds of $(u,v)$ is inspired by \cite{Hollis}. For $0\le \tau<T<\infty$, let $Q_{\tau, T}=\Omega\times (\tau, T)$. For $a\in [1, \infty)$, we denote by $\phi\in L^a(Q_{\tau, T})$
the space of measurable function $\phi:\,Q_{\tau, T}\mapsto\mathbb{R}$ with the norm:
$$
\|\phi\|_{a, \tau, T}=\left(\int_{Q_{\tau, T}} |\phi(x, t)|^a dxdt\right)^{1/a}.
$$

Given $\theta\in L^a(\Q)$, let $\phi$ be the solution of the following backward problem
\begin{equation}\label{back}
\left\{
\begin{array}{lll}
\dis {   \partial_t\phi}=-d_v\Delta\phi-\theta,  \ \ &x\in\Omega, t\in (\tau, T), \medskip \\
\dis  {   \partial_\nu\phi}=0, \ \ &x\in\partial\Omega, t\in (\tau, T), \medskip\\
\phi=0, \ \ &x\in\Omega, t=T.
\end{array}
\right.
\end{equation}
The following lemma comes from \cite[Lemma 3]{Hollis}.
\begin{Lemma}\label{lemma_back}
Let $a\in (1, \infty)$ and $0\le \tau<T<\infty$. Given $\theta\in L^a(\Q)$, let $\phi$ be the solution of \eqref{back}.
Then there exists constant $C(a)$ such that
$$
\|P\phi(\cdot, \tau)\|_{L^a(\Omega)},  \|P\phi\|_{a, \tau, T}, \|\Delta\phi\|_{a, \tau, T}\le C(a)\|\theta\|_{a,\tau, T},
$$
where
$$
P\phi(\cdot, t):=\phi(\cdot, t)-\frac{1}{|\Omega|} \int_\Omega \phi(x, t)dx.
$$
\end{Lemma}

\begin{Lemma}\label{lemma_dual}
For any $a>1$, there exist positive constants $A, B, C$ depending on the initial data such that for any $0\le \tau<T<\infty$, the solution $(u, v)$ satisfies
\begin{equation}\label{dual}
\| v \|_{a, \tau, T}\le   A (T-\tau)^{\frac{1}{a}} + B \|v(\cdot, \tau)\|_{L^a(\Omega)} + C.
\end{equation}
Moreover, there exists $\tau_0>0$ such that \eqref{dual} holds with $A, B, C$ depending only on $N$ for $\tau_0\le \tau<T<\infty$.
\end{Lemma}
\begin{proof}
Let $a'>1$ such that $1/a+1/a'=1$. Let $\theta\in L^{a'}(\Q)$ with $\theta\geq0$ be given and $\phi$ be the solution of \eqref{back}.

Multiplying the two equations of \eqref{density1} by $\phi$, and adding them up and integrating over $Q_{\tau, T}$, we obtain
$$
\int_{\Q} (\partial_t u-d_u\Delta u)\phi dxdt+\int_{\Q} (\partial_t v-d_v\Delta v)\phi dxdt \le 0.
$$
Integrating by parts, we further have
\begin{eqnarray*}
\int_{\Q} (-\partial_t \phi-d_u\Delta \phi)u dxdt+\int_{\Q} (-\partial_t \phi-d_v\Delta \phi)v dxdt \\
+\int_\Omega (u(x, t)+v(x, t))\phi(x, t)dx|_{t=\tau}^{t=T}\le 0.
\end{eqnarray*}
This, together with \eqref{back}, yields
\begin{equation}
\begin{array}{lll}
\label{ee}
\int_{\Q} (u+v)\theta dxdt &\le& \int_\Omega u(x, \tau)\phi(x, \tau)dxdt+ \int_\Omega v(x, \tau)\phi(x, \tau)dx\medskip\\
&&+(d_u-d_v)\dis\int_{\Q} u\Delta \phi dxdt \\
&=:&\mathcal{I}_1+\mathcal{I}_2+ (d_u-d_v)\mathcal{I}_3. \medskip
\end{array}
\end{equation}

On the other hand, integrating the first equation of \eqref{back} on $\Q$, we obtain
\begin{equation}\label{theta}
\int_\Omega \phi(x, \tau) dx = \int_{\Q} \theta(x, t)dxdt.
\end{equation}
As a result, by \eqref{theta}, Lemmas \ref{lemma_s1} and \ref{lemma_back}, and H\"{o}lder's inequality, we have
\begin{eqnarray}
\mathcal{I}_1&=& \frac{1}{|\Omega|} \int_\Omega \phi(x, \tau)dx \int_\Omega u(x, \tau)dx +  \int_\Omega u(x, \tau) (P\phi)(x, \tau) dx \nonumber\\
&\le&  \frac{N}{|\Omega|}   \int_{\Q} \theta(x, t)dxdt + \|u(\cdot, \tau)\|_{L^a(\Omega)}\|P\phi\|_{L^{a'}(\Omega)}  \nonumber\\
&\le&  \frac{N}{|\Omega|} (T-\tau)^{\frac{1}{a}} \|\theta\|_{a',\tau, T} + M_a C(a') \|\theta\|_{a',\tau, T},  \label{i1}
\end{eqnarray}
where $M_a$ is defined in Lemma \ref{lemma_s1}.

Similarly, making use of \eqref{theta} and Lemma \ref{lemma_back}, we get
\begin{eqnarray}
\mathcal{I}_2 &=& \frac{1}{|\Omega|} \int_\Omega \phi(x, \tau)dx \int_\Omega v(x, \tau)dx +  \int_\Omega v(x, \tau) (P\phi)(x, \tau) dx \nonumber\\
&\le&  \frac{N}{|\Omega|}   \int_{\Q} \theta(x, t)dxdt + \|v(\cdot, \tau)\|_{L^a(\Omega)}\|P\phi\|_{L^{a'}(\Omega)}  \nonumber\\
&\le&  \frac{N}{|\Omega|} (T-\tau)^{\frac{1}{a}} \|\theta\|_{a',\tau, T} + C(a')  \|v(\cdot, \tau)\|_{L^a(\Omega)} \|\theta\|_{a',\tau, T}.  \label{i2}
\end{eqnarray}
By Lemmas \ref{lemma_s1}-\ref{lemma_back} and H\"{o}lder inequality, we also have
\begin{eqnarray}
\mathcal{I}_3 &\le& \int_\tau^T \|\Delta\phi(\cdot, t)\|_{L^{a'}(\Omega)} \|u(\cdot, t)\|_{L^{a}(\Omega)} dt \nonumber \\
&\le& M_a\int_\tau^T \|\Delta\phi(\cdot, t)\|_{L^{a'}(\Omega)} dt \nonumber\\
&\le&  M_a \left( \int_\tau^T \|\Delta\phi(\cdot, t)\|_{L^{a'}(\Omega)} ^{a'} dt \right)^{\frac{1}{a'}} \left(\int_\tau^T 1^a dt \right)^\frac{1}{a} \nonumber \\
&\le&  M_a C(a') (T-\tau)^{\frac{1}{a}}  \|\theta\|_{a',\tau, T}.
\label{i3}
\end{eqnarray}

Combining \eqref{ee} and \eqref{i1}-\eqref{i3}, we have
\begin{eqnarray}
\int_{\Q} v\theta dxdt&\leq&\int_{\Q} (u+v)\theta dxdt\nonumber \\
&\le&\left( A (T-\tau)^{\frac{1}{a}} + B \|v(\cdot, \tau)\|_{L^a(\Omega)} + C \right) \|\theta\|_{a',\tau, T},
\label{qq}
\end{eqnarray}
where $A= 2N/|\Omega|+  M_a C(a')$, $B=C(a')$ and $C=M_a C(a') $. Since $\theta\in L^{a'}(\Omega)$ is arbitrary, by \eqref{qq} and duality, $v\in L^a(\Q)$ and \eqref{dual} holds.

By virtue of the above estimates for $\mathcal{I}_i\,(i=1, 2, 3)$, we now use Lemma \ref{lemma_s1} to conclude that there exists $\tau_0>0$ such that all the $M_a$ in the previous inequalities can be replaced by $N_a$ which depends only on $N$ for $\tau_0\le \tau<T$. Therefore, \eqref{dual} holds for $\tau_0\le \tau<T$ with $A, B$ and $C$ depending only on $N$.
\end{proof}

\vskip5pt
With the aid of Lemma \ref{lemma_dual}, one can apply an argument similar to
the proof of \cite[Lemma 7]{Hollis} to establish the following result.

\begin{Lemma}\label{lemma_t}
 For any $a>1$, there exist positive constants $\Lambda_a,\Gamma_a$ and $\tilde M_a$ depending on the initial data and a sequence $\{t_k\}_{k=0}^\infty$ with $t_0=0$ such that the solution $(u, v)$ satisfies
\begin{enumerate}
\item[\rm(i)] $1\le t_{k+1}-t_k\le \Lambda_a;$
\item[\rm(ii)] $\|v(\cdot, t_k)\|_{L^a(\Omega)}\le \tilde M_a;$
\item[\rm(iii)]  $\int_{t_k}^{t_{k+1}} \|v(\cdot, t)\|_{L^a(\Omega)}^a dt \le \Gamma_a$.
\end{enumerate}
Moreover, there exists $\tau_0>0$ such that {\rm(i)-\rm(iii)} hold with $t_0=0$ replaced by $t_0=\tau_0$ and $\Lambda_a,\Gamma_a$ and $\tilde M_a$ depending only on $N$.
\end{Lemma}

Our main result of this subsection on the uniform bounds follows from a semigroup computation.

\begin{Theorem}\label{Theorem_bound0}
{   Suppose that $\beta>0$, $\gamma, \mu\ge 0$, and {\rm (H1)-(i)} holds.}  Let $(u, v)$ be a  nonnegative classical solution of \eqref{density1}. Then there exists $M_\infty>0$ depending on the initial data such that \eqref{Minfdd} holds.
Moreover, there exists $N_\infty$ independent of initial data such that \eqref{Ninfdd} holds for any  global classical solution $(u, v)$.
\end{Theorem}

\begin{proof}
Let $T(t)$ be the semigroup in $X:=L^a(\Omega)\,(a>1)$ generated by $$\mathcal{A}:=d_I\Delta-1$$ with the domain
$$
D(\mathcal{A})=\left\{u\in W^{2, a}(\Omega): \ {   \partial_\nu u}=0 \text{ on } \partial\Omega \right\}.
$$
Let $X_\alpha\, (0<\alpha<1)$ be the fractional power space with graph norm. Choose $a$ large and $\alpha$ such that $\alpha>n/(2a)$ and $\alpha<1-1/a$. Then $X_\alpha\subset C(\bar\Omega)$. It is well known that there exists $M_\alpha>0$ such that
$$
\| \mathcal{A}^\alpha T(t)\|\le \frac{M_\alpha}{t^\alpha}, \ \ \forall t>0.
$$

Let $\{t_k\}$ be the sequence given in Lemma \ref{lemma_t} (with $a$ replaced by $ap$). By the second equation of \eqref{density1}, for any $t\in (t_k, t_{k+1})$, we have
\begin{eqnarray}\label{T}
\nonumber v(\cdot, t)&=&T(t-t_{k-1})v(\cdot, t_{k-1}) \\
&&+ \int_{t_{k-1}}^t T(t-\tau)\big[ \beta u^q(\cdot, \tau)v^p(\cdot, \tau)+(1-\gamma-\mu)u^s(\cdot, \tau)v^r(\cdot, \tau) \big]d\tau.
\end{eqnarray}
By \eqref{T} and Lemma \ref{lemma_t}, for any $t\in (t_k, t_{k+1})$, we have
\begin{eqnarray}
\|\mathcal{A}^\alpha v(\cdot, t)\|_{L^a(\Omega)} &\le&  \|\mathcal{A}^\alpha T(t-t_{k-1}) v(\cdot, t_{k-1})\|_{L^a(\Omega)}
\nonumber \\
&&+ \int_{t_{k-1}}^t \|\mathcal{A}^\alpha T(t-\tau) [\beta u^q(\cdot, \tau)v^p(\cdot, \tau)+{(1-\gamma-\mu)}u^s(\cdot, \tau)v^r(\cdot, \tau)]\|_{L^a(\Omega)} d\tau \nonumber\\
&\le& \frac{M_\alpha  \|v(\cdot, t)\|_{L^a(\Omega)}}{(t-t_{k-1})^\alpha} \nonumber \\
&&+\int_{t_{k-1}}^t M_\alpha\frac{{  \beta}\|u^q(\cdot, \tau)v^p(\cdot, \tau)\|_{L^a(\Omega)}+{|1-\gamma-\mu|} \|u^s(\cdot, \tau)v^r(\cdot, \tau)\|_{L^a(\Omega)}}{(t-\tau)^\alpha}d\tau \nonumber\\
&\le& M_\alpha\tilde M_a +  M_\alpha{  \beta}\int_{t_{k-1}}^{t} \frac{\|u^q(\cdot, \tau)\|_{L^a(\Omega)}  \|v^p(\cdot, \tau)\|_{L^a(\Omega)}}{(t-\tau)^\alpha}d\tau
\nonumber\\
&& + M_\alpha {|1-\gamma-\mu|}\int_{t_{k-1}}^{t} \frac{\|u^s(\cdot, \tau)\|_{L^a(\Omega)}\|v^r(\cdot, \tau)\|_{L^a(\Omega)}}{(t-\tau)^\alpha}d\tau \nonumber\\
&=:&M_\alpha\tilde M_a +  M_\alpha{ \beta}\mathcal{I}_1+M_\alpha {|1-\gamma-\mu|} \mathcal{I}_2. \nonumber
\end{eqnarray}
By Lemma \ref{lemma_s1}, it holds
$$
\|u^q(\cdot, \tau)\|_{L^a(\Omega)}\le  \|u(\cdot, \tau)\|_{L^{aq}(\Omega)}^q\le M_{aq}^q.
$$
Noticing that $\|v^p(\cdot, \tau)\|_{L^a(\Omega)} \le \|v(\cdot, \tau)\|_{L^{ap}(\Omega)}^p$, and using H\"{o}lder's inequality with $1/a'+1/a=1$, we deduce
\begin{eqnarray}
\mathcal{I}_1 &\le&  M_{aq}^q \left( \int_{t_{k-1}}^t \frac{1}{(t-\tau)^{\alpha a'}}d\tau \right)^{\frac{1}{a'}} \left(\int_{t_{k-1}}^t  \|v(\cdot, \tau)\|_{L^{ap}(\Omega)}^{ap}d\tau\right)^{\frac{1}{a}}  \nonumber\\
&\le& M_{aq}^q \left( \frac{(t-t_{k-1})^{1-\alpha a'}}{1-\alpha a'} \right)^\frac{1}{a'}   \left(\int_{t_{k-1}}^{t_k}  \|v(\cdot, \tau)\|_{L^{ap}(\Omega)}^{ap} d\tau+ \int_{t_{k}}^t  \|v(\cdot, \tau)\|_{L^{ap}(\Omega)}^{ap}d\tau\right)^{\frac{1}{a}}  \nonumber\\
&\le& M_{aq}^q \frac{(2\Lambda_{ap})^\frac{1-\alpha a'}{a'}}{(1-\alpha a')^\frac{1}{a'}}  (2\Gamma_{ap})^\frac{1}{a},\nonumber
\end{eqnarray}
where we have used Lemma \ref{lemma_t} and the fact that $\alpha a'=\alpha a/ (a-1)<1$ as $\alpha<1-1/a$.

Similarly, we can obtain a similar estimate for $\mathcal{I}_2$.
Therefore, there exists a positive constant $C$ such that
\begin{equation}\label{semig}
 \|A^\alpha v(\cdot, t)\|_{L^a(\Omega)}\le C,\ \ \ \mbox{for all $t\ge t_1$}.
\end{equation}
By the embedding $X_\alpha\subset C(\bar\Omega)$, there exists $M_\infty>0$ such that $\|v(\cdot, t)\|_{L^\infty(\Omega)} \le M_\infty$ for $t\ge 0$. Similarly, we can use Lemmas \ref{lemma_s} and \ref{lemma_t} to prove that $\|u(\cdot, t)\|_{L^\infty(\Omega)}\le M_\infty$ for all $t\ge 0$. This proves \eqref{Minfdd}.

In light of Lemmas \ref{lemma_s1}-\ref{lemma_t}, a similar semigroup argument allows one to assert \eqref{Ninfdd};
the details are omitted here.
\end{proof}

\begin{Theorem}\label{Theorem_boundp} {   Suppose that  $\beta>0$, $\gamma, \mu\ge 0$, and \rm (H1)-(ii)} holds. Let $(u, v)$ be a  nonnegative  classical solution of \eqref{density1}. Then there exists $M_\infty>0$ depending on the initial data such that \eqref{Minfdd} holds.
If in addition $p=r=1$, there exists $N_\infty$ independent of initial data such that \eqref{Ninfdd} holds for any  global classical solution $(u, v)$.
\end{Theorem}
\begin{proof} Since $p=r$ and $s<q$, from the first equation of \eqref{density1} we have
$$
u_t-d_u\Delta u=v^ru^s(-\beta u^{q-s}+\gamma).
$$
By the maximum principle, there exists $M>0$ depending on the initial data such that $\|u(\cdot, t)\|_{L^\infty(\Omega)}\le M$ for all $t\ge 0$. The proof of the $L^\infty$-bounds of $v$ is similar to Lemmas \ref{lemma_dual}-\ref{lemma_t} and Theorem \ref{Theorem_bound0}.

If $p=r=1$ and $s<q$, then given $a\geq1$, for any $\epsilon>0$, we have
\begin{equation}\label{pr1}
u^{s+a-1}v \le \epsilon u^{q+a-1}v + A_1 v,\ \ \forall u,\,v\geq0,
\end{equation}
for some $A_1=A_1(\epsilon)$. Using this inequality instead of \eqref{pq} and the fact $\int_\Omega v(x, t)dx\le N$, we can prove the uniform boundedness result for $u$ stated as in Lemma \ref{lemma_s}. To see this, for any $k\geq0$, we can choose $\epsilon$ so small that \eqref{pr1} implies
$$
- \int_\Omega\beta u^{2^k-1+q}v^pdx+\int_\Omega \gamma u^{2^k-1+s}v^r dx \le  A_1 N |\Omega|,\ \ \forall t\geq0.
$$
Multiplying both sides of the first equation of \eqref{density1} by $u^{2^k-1}$ and integrating over $\Omega$, we obtain
\begin{eqnarray*}
\frac{1}{2^k}\frac{d}{dt}\int_\Omega u^{2^k} dx &\le& -\frac{2^k-1}{2^{2k-2}} d_u\int_\Omega |\nabla  u^{2^{k-1}}|^2dx + \int_\Omega (-\beta u^{q}v^p+ \gamma u^{s}v^r)u^{2^k-1} dx \\
&\le & -2C_1 \int_\Omega u^{2^k} dx + C_3\left(\int_\Omega  u^{2^{k-1}}dx\right)^2 + A_1N|\Omega|,\ \ \forall t\geq0.
\end{eqnarray*}
Here \eqref{xii} is used as in the proof of Lemma \ref{lemma_s}. Therefore, a similar induction argument as in Lemma \ref{lemma_s} gives the $L^\infty$-bounds of $u$ as in Lemma \ref{lemma_s}. The rest of the proof is similar to Theorem \ref{Theorem_bound0}, as one can follow the same arguments to establish similar results in Lemmas \ref{lemma_s1}-\ref{lemma_t}.
\end{proof}

\subsection{Proof of Theorem \ref{Theorem_boundaa} under  condition (H2)-(i)}

In this subsection, we establish the $L^\infty$-bounds of the solutions of \eqref{density1}
if (H2)-(i) is fulfilled. We start with the following lemma.

\begin{Lemma} \label{lemma_bb} Assume that {\rm (H2)-(i)} holds and $a>1$. Let
$b=\frac{1-p}{q}(a-1)+1.$
Then for any $\epsilon>0$, there exists $A=A(\epsilon)$ such that the following two inequalities hold:
\begin{equation}\label{aaa}
u^{s+a-1}v^r\le \epsilon u^{q+a-1}v^{p}+ {\color{red} A} v^{b}+1,\ \ \forall u,\, v\ge 0,
\end{equation}
\begin{equation}\label{bbb}
u^qv^{p+b-1}\le \epsilon u^{q+a-1}v^{p}+ {\color{red} A} v^{b},\ \ \forall u,\, v\ge 0.
\end{equation}
\end{Lemma}
\begin{proof} We first prove \eqref{aaa}. Obviously, \eqref{aaa} is true when $u,\,v\leq1$. It suffices to consider the case of
either $u\geq1$ or $v\geq1$. In the sequel, we only handle the case $v\geq1$; the other case can be treated similarly.

It is easily noticed that
 $$
 u^{s+a-1}v^r\le \epsilon u^{q+a-1}v^{p}\ \ \ \ \mbox{if}\ \ \epsilon^{-\frac{1}{q-s}} v^\frac{r-p}{q-s}\le u,
 $$
and
 $$
 u^{s+a-1}v^r\le A v^b\ \ \ \ \mbox{if}\ \ 0\leq u\le A^\frac{1}{s+a-1} v^\frac{b-r}{s+a-1}.
 $$
Due to $v\ge 1$, \eqref{aaa} follows readily if we choose $A=\epsilon^{-\frac{s+a-1}{q-s}}$ and assume that
$$
\frac{r-p}{q-s}\le \frac{b-r}{s+a-1}.
$$
The latter inequality is satisfied if
$$
\frac{r-p}{q-s}\le \frac{1-p}{q} \ \ \ \text{and}  \ \ \  \frac{b-1}{a-1} \le \frac{b-r}{s+a-1}.
$$
These two inequalities can be checked directly using $ps-qr\ge q-s$ and $\frac{1-p}{q}=\frac{b-1}{a-1}$.
Thus, the above analysis verifies \eqref{aaa}.

Clearly, we have
 $$
 u^qv^{p+b-1}\le \epsilon u^{q+a-1}v^p,\ \ \ \forall u\geq0,\ \ \mbox{if}\ 0\leq v\le \epsilon^{1/{(b-1)}}u^{(a-1)/(b-1)},
 $$
and
 $$
 u^qv^{p+b-1}\le Av^{b},\ \ \ \forall u\geq0,\ \ \mbox{if}\  v\ge A^{-1/(1-p)}u^{q/(1-p)}.
 $$
By taking $A=\epsilon^{(p-1)/(b-1)}$ and the definition of $b$, we obtain \eqref{bbb} .
\end{proof}

{   For Lemmas \ref{lemma_sii}-\ref{Lemma_rp1}, we suppose that $\beta>0$, $\gamma, \mu\ge 0$, and {\rm (H2)-(i)} holds. Let $(u, v)$ be the nonnegative
 classical solution of \eqref{density1}. }
\begin{Lemma}\label{lemma_sii}
 Let $a>1$, and
$b=\frac{1-p}{q}(a-1)+1$. Then the following statements hold:

\begin{enumerate}
\item[\rm(i)] If there exists $M_1>0$ depending on the initial data such that
$$
\int_\Omega u^\frac{a}{2} dx, \ \  \int_\Omega v^\frac{b}{2} dx \le M_1,\ \ \forall t\ge 0,
$$
then there exists $M_2>0$  depending on the initial data such that
 $$
 \int_\Omega u^a dx, \ \  \int_\Omega v^b dx \le M_2,\ \ \forall t\ge 0.
 $$

\item[\rm(ii)]  If there exists $N_1>0$ independent of initial data such that
$$
\limsup_{t\rightarrow\infty} \int_\Omega u^\frac{a}{2} dx, \ \ \limsup_{t\rightarrow\infty} \int_\Omega v^\frac{b}{2} dx  \le N_1,
$$
then there exists $N_2>0$ independent of initial data such that
$$
\limsup_{t\rightarrow\infty} \int_\Omega u^a dx, \ \ \limsup_{t\rightarrow\infty} \int_\Omega v^b dx\le N_2.
$$
\end{enumerate}
\end{Lemma}
\begin{proof}
{   Let $\sigma_0=\min\{\beta, \gamma\}$ and $\sigma^0=\max\{\beta, \gamma+\mu\}$.} Multiplying both sides of the first equation of \eqref{density1} by $u^{a-1}$ and integrating over $\Omega$, we obtain
\begin{equation}\label{eS1}
\frac{1}{a} \frac{d}{dt}\int_\Omega u^a dx \le -C_1\int_\Omega |\nabla u^\frac{a}{2}|^2dx-\sigma_0\int_\Omega u^{q+a-1}v^p dx+\sigma^0\int_\Omega u^{s+a-1} v^r dx,\ \ \forall t\ge 0,
\end{equation}
where $C_1={4d_u(a-1)}/{a^2}$.

An application of \eqref{xii} with $\epsilon_*=C_1$ and $w=u^{a/2}$ gives
\begin{equation}\label{eS2}
-C_1\int_\Omega |\nabla u^\frac{a}{2}|^2dx \le C_2 \left(\int_\Omega u^\frac{a}{2}dx\right)^2 - \int_\Omega u^adx,\ \ \forall t\ge 0,
\end{equation}
for some constant $C_2>0$.

By \eqref{aaa} of Lemma \ref{lemma_bb} with $\epsilon=\sigma_0/2\sigma^0$, there exists $A_1>0$ such that
\begin{equation}\label{eS3}
\int_\Omega u^{s+a-1}v^r dx\le \frac{\sigma_0}{2\sigma^0}\int_\Omega u^{q+a-1}v^{p}dx +A_1\int_\Omega v^b dx+|\Omega|,\ \ \forall t\ge 0.
\end{equation}

Combining \eqref{eS1}-\eqref{eS3}, we have
\begin{eqnarray}\label{TTT}
\nonumber\frac{1}{a} \frac{d}{dt}\int_\Omega u^a dx &\le&  \sigma^0 A_1\int_\Omega v^b dx +C_2 \left(\int_\Omega u^\frac{a}{2}dx\right)^2 \\
&&- \int_\Omega u^adx -\frac{\sigma_0}{2}\int_\Omega u^{q+a-1}v^p dx+\sigma^0|\Omega|,\ \ \forall t\ge 0.
\end{eqnarray}

Multiplying both sides of the second equation of \eqref{density1} by $v^{b-1}$ and integrating over $\Omega$, we obtain
\begin{equation}\label{eS4}
\frac{1}{b} \frac{d}{dt}\int_\Omega v^b dx \le -C_3\int_\Omega |\nabla v^\frac{b}{2}|^2dx+\sigma^0\int_\Omega u^{q}v^{p+b-1} dx-\sigma_0\int_\Omega u^sv^{r+b-1} dx,\ \ \forall t\ge 0,
\end{equation}
where $C_3={4d_v(b-1)}/{b^2}$.

Using \eqref{bbb} of Lemma \ref{lemma_bb} with $\epsilon=\sigma_0/2\sigma^0$, we can find $A_2>0$ such that
\begin{equation}\label{eS5}
\int_\Omega u^{q}v^{p+b-1} dx\le \frac{\sigma_0}{2\sigma^0}\int_\Omega u^{a+q-1}v^p dx +A_2\int_\Omega v^b dx,\ \ \forall t\ge 0.
\end{equation}
By \eqref{xii} with $\epsilon_*=\frac{C_3}{\sigma^0(A_1+A_2+1)}$ and $w=v^{b/2}$, there exists $C_4>0$ such that
\begin{equation}\label{eS6}
-C_3\int_\Omega |\nabla v^\frac{b}{2}|^2dx \le C_4 \left(\int_\Omega v^\frac{b}{2}dx\right)^2 - \sigma^0(A_1+A_2+1) \int_\Omega v^bdx,\ \ \forall t\ge 0.
\end{equation}

Because of \eqref{eS4}-\eqref{eS6}, we have
\begin{equation}\label{eS7}
\frac{1}{b} \frac{d}{dt}\int_\Omega v^b dx\le C_4 \left(\int_\Omega v^\frac{b}{2}dx\right)^2 -(\sigma^0 A_1+1)\int_\Omega v^bdx + \frac{\sigma_0}{2}\int_\Omega u^{a+q-1}v^p dx,\ \ \forall t\ge 0.
\end{equation}
Hence, from \eqref{TTT} and \eqref{eS7} it follows that
\begin{eqnarray}\label{TTT1}
\nonumber\frac{1}{a} \frac{d}{dt}\int_\Omega u^a dx+\frac{1}{b} \frac{d}{dt}\int_\Omega v^b dx &\le& \sigma^0|\Omega|+ C_2 \left(\int_\Omega u^\frac{a}{2}dx\right)^2  \\
&& + C_4 \left(\int_\Omega v^\frac{b}{2}dx\right)^2-\int_\Omega u^adx-\int_\Omega v^bdx,\ \forall t\ge 0.
\end{eqnarray}
This readily yields our claimed statements.
\end{proof}

\begin{Lemma}\label{Lemma_rp1}
For any $k\geq1$, there exists $M_k>0$ depending on the initial data such that the solution $(u, v)$ satisfies
$$
\|u(\cdot, t)\|_{L^k(\Omega)},  \|I(\cdot, t)\|_{L^k(\Omega)} \le M_k, \ \ \forall t\ge 0.
$$
Moreover, there exists $N_k$ independent of initial data such that, for any  global classical solution $(u, v)$, the following hold:
\begin{equation}\label{Nr}
\limsup_{t\rightarrow\infty}\|u(\cdot, t)\|_{L^k(\Omega)},  \ \limsup_{t\rightarrow\infty}\|v(\cdot, t)\|_{L^k(\Omega)} \le N_k.
\end{equation}
\end{Lemma}
\begin{proof} Suppose that $(1-p)/q\le 1$. We proceed by induction.
If $k=1$, our lemma follows from the fact $\int_\Omega (u+v) dx\le N,\ \forall t\geq0$. Assume that the desired result holds for $k-1$ with $k\ge 2$. Let
$$
a=k,\ \ b=(1-p)(k-1)/q+1.
$$
Then, we have $a/2, b/2\le k-1$ and hence Lemma \ref{lemma_sii} ensures that the result holds for $k$.

Suppose that $(1-p)/q> 1$. This result can be still proved by an induction analysis as above;
the only difference is that we now set $a=q(k-1)/(1-p)+1$ and so $b:=(1-p)(a-1)/q+1=k$ in Lemma \ref{lemma_sii}.
\end{proof}

\vskip5pt
Based on Lemma \ref{Lemma_rp1}, similar to Theorem \ref{Theorem_bound0},
we can use a semigroup method to establish the following uniform bounds.

\begin{Theorem}\label{Theorem_bound2}
{   Suppose that  $\beta>0$, $\gamma, \mu\ge 0$, and {\rm (H2)-(i)} holds.}  Let $(u, v)$ be a  nonnegative
 classical solution of \eqref{density1}. Then there exists $M_\infty>0$ depending on the initial data such that
 \eqref{Minfdd} holds. Moreover, there exists $N_\infty$ independent of initial data such that \eqref{Ninfdd} holds for any  global classical solution $(u, v)$.
\end{Theorem}

Theorem \ref{Theorem_boundaa} is a combination of Theorems \ref{Theorem_bound0}, \ref{Theorem_boundp} and \ref{Theorem_bound2}.

\section{Proof of Theorem  \ref{th-a}}
We first prove the following result about the local existence and positivity of the solutions of \eqref{density}.
\begin{Lemma}\label{lemma_local}
Assume that {\rm (A1)-(A2)} hold. Then \eqref{density} has a unique solution $(S, I)$ on $\bar\Omega\times [0, T_{max})$, where $T_{max}\le \infty$ is the maximal time for the existence of solution. Moreover, $(S, I)$ satisfies
$$
S(x, t), I(x, t)>0 \ \ \text{for all} \ \ (x,t)\in\bar\Omega\times(0, T_{max}),
$$
and if $T_{max}<\infty$, then
$$
\lim_{t\rightarrow T_{max}} \|S(\cdot, t)\|_{L^\infty(\Omega)}+\|I(\cdot, t)\|_{L^\infty(\Omega)}=\infty.
$$
\end{Lemma}
\begin{proof} {   By (A2), the right-hand side of \eqref{density} is continuously differentiable at $(S, I)=(S_0, I_0)$.} By standard theory for parabolic equations, \eqref{density} has a unique nonnegative classical solution $(S(x, t), I(x, t))$ on $\bar\Omega\times [0, \hat t]$ for some $\hat t>0$.

To see the positivity of $S$ and $I$,  we first observe that $I$ is a supersolution to the initial-boundary value problem:
 \begin{equation}
 \left\{ \begin{array}{llll}
 \dis {   \partial_t w}-d_I\Delta w=-2\sigma^0w,&x\in\Omega,\,t\in (0, \hat t],\medskip\\
 \dis {   \partial_\nu w}=0,&x\in\partial\Omega,\,t\in(0, \hat t],\medskip\\
 \dis w(x,0)= I_0(x),&x\in\Omega,
 \end{array}\right.
 \label{2-a}
 \end{equation}
Let $\underline w$ be the solution of \eqref{2-a}. From the well-known strong maximal principle and
Hopf boundary lemma for parabolic equations, we have $\underline w>0$ for all $(x,t)\in\bar\Omega\times(0,\hat t]$. Thus, the parabolic comparison principle ensures $I(x,t)\geq\underline w(x,t)>0$ for all  $(x,t)\in\bar\Omega\times(0, \hat t]$.

It remains to show the positivity of $S$. If $q\geq1$,
as the reaction term $-\beta S^qI^p+\gamma I$ is Lipschitz with respect to $S$, a standard comparison analysis yields that
$$
S(x,t)>0 \ \ \text{for all}  \ (x,t)\in\bar\Omega\times(0, \hat t].
$$
If
$0<q<1$, by (A4), $S$ is a supersolution of the following problem:
 \begin{equation}
 \left\{ \begin{array}{llll}
 \dis {   \partial_t z}-d_S\Delta z=-C^p\sigma^0z^q+\sigma_0 \underline w(x,t),&x\in\Omega,\,(0, \hat t], \medskip \\
 \dis {   \partial_\nu z}=0,&x\in\partial\Omega,\,(0, \hat t], \medskip\\
 \dis z(x,0)= S_0(x),&x\in\Omega,
 \end{array}\right.
 \label{auxi-a}
 \end{equation}
where $C$ is some positive number such that $I(x, t)\le C$ for all $(x, t)\in \bar\Omega\times [0, \hat t]$. Denote by $z$ the unique solution of \eqref{auxi-a}.
By the comparison principle,
$$
S(x,t)\geq z(x,t) \text{ for all } (x,t)\in\bar\Omega\times(0,\hat t).
$$
Suppose to the contrary that there exists $(x_0,t_0)\in\bar\Omega\times(0,\hat t]$ such that $z(x_0,t_0)=0$. If $x_0\in\Omega$, then
$$
{   \partial_t z}(x_0,t_0)\leq0 \ \text{ and } \ \Delta z(x_0,t_0)\geq0.
$$
From this and $ \underline w(x_0,t_0)>0$, we obtain a contradiction using the first equation of \eqref{auxi-a}. If $x_0\in\partial\Omega$, by the nonnegativity of $z$, one can easily apply Hopf boundary lemma for parabolic equations to conclude that
$$
{   \partial_\nu z}(x_0,t_0)<0,
$$
which contradicts the boundary condition in \eqref{auxi-a}. Consequently,
$$
S(x,t)\geq z(x,t)>0 \ \text{ for all }  (x,t)\in\bar\Omega\times(0,\hat t].
$$
Hence, the positivity of $S, I$ on $\bar\Omega\times(0,\hat t]$ guarantees that $S^qI^p$ is locally Lipschitz. Then, it is a standard process to extend the time for the existence of solution to a maximal interval $[0, T_{max})$, where either $T_{max}=\infty$ or the solution blows up at finite time $T_{max}$.
The proof is complete.
\end{proof}

\vskip6pt
Now we can prove the global existence and uniform boundedness of the solutions of \eqref{density}:
\begin{proof}[Proof of Theorem \ref{th-a}]
Note that \eqref{density} is a special case of \eqref{density1} with $s=0$ and $r=1$. The uniform bounds in the cases $p>1$, $p=1$, and $0<p<1$ are covered by Theorems \ref{Theorem_bound0}-\ref{Theorem_boundp} and \ref{Theorem_bound2}, respectively. Thus, our assertions follow from Theorems \ref{Theorem_bound0}-\ref{Theorem_boundp} and \ref{Theorem_bound2}, and Lemma \ref{lemma_local}.
\end{proof}

\section{Proof of Theorem \ref{result1}}
%
%
%
%
%
%

We need the following lemma in order to prove Theorem \ref{result1}.
\begin{Lemma}[{\cite[Lemma 1.1]{Wang-2018}}] \label{lem1} Let $a\geq0$ and $b>0$ be constants. Assume
that $\phi\in C^1([a,\infty))$, $\psi\geq0$, $\phi$ is
bounded from below in $[a,\infty)$, and satisfies
 $$
 \phi'(t)\leq-b\psi(t)+g(t),\ \ \forall t\in[a,\infty),
 $$
where $\int_a^\infty g(t)dt<\infty$. Furthermore, assume that either $\psi\in C^1([a,\infty))$ and
$\psi'(t)\leq K$ on $[a,\infty)$, or $\psi\in C^{\alpha'}([a,\infty))$ for some constant $\alpha'\in(0,1)$
and $\|\psi\|_{C^{\alpha'}([a,\infty))}\leq K$, for some positive constant $K$. Then we have $\lim_{t\to\infty}\psi(t)=0$.
\end{Lemma}

We also recall the well-known Poincar\'{e} inequality.
\begin{Lemma} \label{lem2} The following inequality holds:
 $$
 \lambda_1\int_\Omega|g(x)-\hat g|^2dx\leq\int_\Omega|\nabla g(x)|^2dx,\ \ \forall g\in H^1(\Omega),
 $$
where $\hat g=\frac{1}{|\Omega|}\int_\Omega g(x)dx$, and $\lambda_1$ is the
first positive eigenvalue of the Laplacian operator $-\Delta$ with homogeneous Neumann boundary condition.
\end{Lemma}

We are ready to present the proof of Theorem \ref{result1}.
\begin{proof}[Proof of Theorem \ref{result1}]
Noticing Theorem \ref{th-a}, let $C>0$ be such that $S(x,t)$, $I(x,t)\leq C$ for all  $(x,t)\in\bar\Omega\times(0,\infty)$.
With the help of the well-known parabolic-type $L^p$ and Schauder estimates
and embedding theorems (see, for instance, \cite[Theorems 7.15, 7.20]{Lie}), one can employ
standard argument to conclude that
 \bes
 \|\nabla S(x,\cdot)\|_{C^{\alpha/2}([1,\infty))}+\|\nabla I(x,\cdot)\|_{C^{\alpha/2}([1,\infty))}
 \leq C_0,\ \ \forall x\in\bar\Omega,
 \label{2.2}
 \ees
 \bes
 \|S(x,\cdot)\|_{C^{(\alpha+1)/2}([1,\infty))}
 +\|I(x,\cdot)\|_{C^{(\alpha+1)/2}([1,\infty))}\leq C_0,\ \ \forall x\in\bar\Omega,
 \label{2.3}
 \ees
and
 \bes
 \|S(\cdot,t)\|_{C^{1+\alpha}(\bar\Omega)}
 +\|I(\cdot,t)\|_{C^{1+\alpha}(\bar\Omega)}\leq C_0,\ \ \forall t\geq1.
 \label{2.4}
 \ees
For this, one may refer to \cite[Theorems 2.2, 2.3]{Wang-2018} and \cite[Theorem A2]{BDG}).
Here the positive constant $C_0$ is independent of $S,\,I$ and $t\geq1$.

Integrating both equations of \eqref{density} over $\Omega$ and adding the resulting identities, we obtain
\begin{equation}\label{2.5} \ \
\frac{d}{dt}\int_\Omega (S(x,t)+I(x,t))dx=-\int_\Omega \mu(x,t)I(x,t)dx
\leq-\sigma_0\int_\Omega I(x,t)dx<0,\ \forall t>0.
\end{equation}
This implies that $\int_\Omega (S(x,t)+I(x,t))dx$ is decreasing with respect to the time $t\geq0$.
Thus, the limit $\lim_{t\to\infty}\int_\Omega (S(x,t)+I(x,t))dx$ exists.
By taking
 $$
 \phi(t)=\int_\Omega S(x,t)+I(x, t)dx,\ \ \psi(t)=\int_\Omega I(x,t)dx,\ \ g(t)=0,\ \ t\geq1
 $$
in Lemma \ref{lem1}, combined with \eqref{2.3} and \eqref{2.5}, it follows that
 $$
 \int_\Omega I(x,t)dx\to0,\ \ \mbox{as}\ t\to\infty.
 $$
In view of \eqref{2.4} and the standard embedding theorem, it is necessary that
 \bes
 \label{I-converge}
 I(\cdot,t)\to0 \ \ \mbox{uniformly on}\ \bar\Omega,\ \ \mbox{as}\ t\to\infty.
 \ees
As $\lim_{t\to\infty}\int_\Omega (S(x,t)+I(x,t))dx$ exists, \eqref{I-converge} indicates that
$\lim_{t\to\infty}\int_\Omega S(x,t)dx$ also exists. Hence, we may assume that
 \bes
 \label{S-converge-1}
 \frac{1}{|\Omega|}\int_\Omega S(x,t)dx\to S_*,\ \ \mbox{as}\ t\to\infty
 \ees
for some constant $S_*\geq0$.

Next, we are going to determine the limit of the component $S$ as $t\to\infty$. Multiplying the first equation in \eqref{density} by $S$ and then integrating over $\Omega$ yield
 \begin{equation}
 \begin{aligned}
 \label{2.6}
 \frac{1}{2}\frac{d}{dt}\int_\Omega S^2(x,t)dx&=-d_S\int_\Omega|\nabla S(x,t)|^2dx
 -\int_\Omega \beta(x,t)S^{q+1}(x,t)I^p(x,t)dx\\
 &\ \ \ \ +\int_\Omega \gamma(x,t)S(x,t)I(x,t)dx\\
 &\leq-d_S\int_\Omega|\nabla S(x,t)|^2dx+C\sigma^0\int_\Omega I(x,t)dx,\ \ \forall t>0.
 \end{aligned}
 \end{equation}
In Lemma \ref{lem1}, we now set
 $$
 \phi(t)=\frac{1}{2}\int_\Omega S^2(x,t)dx,\ \ \psi(t)
 =d_S\int_\Omega|\nabla S(x,t)|^2dx,\ \ g(t)=C\sigma^0\int_\Omega I(x,t)dx,\ \ t\geq1.
 $$
Thanks to \eqref{2.2}, $\phi,\,\psi$ satisfy the conditions in Lemma \ref{lem1}.

In order to apply Lemma \ref{lem1}, it remains to verify $\int_1^\infty g(t)dt<\infty$.
In fact, integrating \eqref{2.5} from $0$ to $\infty$ with respect to $t$, we deduce that
\begin{eqnarray}
\int_\Omega (S(x,t)+I(x,t))dx\leq\int_\Omega (S_0(x)+I_0(x))dx-\sigma_0\int_0^\infty\int_\Omega I(x,t)dx,
\nonumber
\end{eqnarray}
which in turn yields
\begin{eqnarray}
\int_1^\infty g(t)dt=\int_1^\infty\int_\Omega I(x,t)dx\leq\frac{1}{\sigma_0}\int_\Omega (S_0(x)+I_0(x))dx<\infty.
\nonumber
\end{eqnarray}
Therefore, Lemma \ref{lem1} ensures
 $$
 \int_\Omega|\nabla S(x,t)|^2dx\to0,\ \ \mbox{as}\ t\to\infty.
 $$
This, combined with Lemma \ref{lem2} and \eqref{S-converge-1}, immediately infers that
 \bes
 \nonumber
 \int_\Omega|S(x,t)-S_*|^2dx\to0 ,\ \ \mbox{as}\ t\to\infty.
 \ees
Then, by \eqref{2.4} and the standard embedding theorem, we have
\bes
 \label{S-converge-2}
 S(x,t)\to S_*\ \ \mbox{uniformly on}\ \bar\Omega ,\ \ \mbox{as}\ t\to\infty.
 \ees

In the following, we will show $S_*=0$ if $p<1$, and $S_*>0$ if $p\geq1$.

We first consider the case of $p<1$. We proceed indirectly by supposing that $S_*>0$.
Then, thanks to \eqref{I-converge}, \eqref{S-converge-2} and our assumption (A3),
there is a large number $T_0>0$ such that $\beta S^q-(\gamma+\mu)I^{1-p}>0$
for all $(x,t)\in\bar\Omega\times[T_0,\infty)$. Thus, $I$ satisfies
 $$
 {   \partial_t I}-d_I\Delta I\geq I^p[\beta S^q-(\gamma+\mu)I^{1-p}]>0,\ \ \text{for all } x\in\Omega,\,t\geq T_0.
 $$
As a result, a simple comparison analysis guarantees $I(x,t)\geq \min_{\bar\Omega}I(x,T_0)>0$ for all $(x,t)\in\bar\Omega\times[T_0,\infty)$, contradicting with \eqref{I-converge}. This shows that $S_*=0$.

\vskip6pt
We now verify that $S_*>0$ if $p\ge 1$.
We argue by contradiction again and suppose that $S_*=0$. So it holds
 \bes
 (S,I)\to(0,0)\ \ \mbox{ uniformly on}\ \bar\Omega,\ \mbox{ as}\ t\to\infty.
 \label{false}
 \ees
Due to $p\geq1$, one can find a large $T_1>0$ such that
$$\beta(x,t)S^q(x,t)-(\gamma(x,t)+\mu(x,t))\leq-\sigma_0/2,\ \ I^p(x,t)\leq I(x,t)$$ for all
$(x,t)\in\bar\Omega\times[T_1,\infty)$. Consequently, $I$ is a subsolution to the following ODE problem:
 \begin{equation}
 \left \{\begin{array}{lll}
 {   w'}=-\frac{1}{2}\sigma_0w,\ t>T_1,\medskip\\
 \dis w(T_1)=\max_{x\in\bar\Omega}I(x,T_1)>0.
 \end{array}\right.
 \label{auxi-2}
 \end{equation}
Thus, it holds
 \bes
 \label{2.7}
 I(x,t)\leq w(t)=\max_{x\in\bar\Omega}I(x,T_1)e^{-\frac{1}{2}\sigma_0(t-T_1)},\ \ \forall (x,t)\in\bar\Omega\times[T_1,\infty).
 \ees

We have to distinguish two different cases: $q\geq1$ and $0<q<1$. In the latter case, we need (A4)-(ii), i.e., $\gamma\geq\sigma_0$ on $\bar\Omega\times[0,\infty)$.

We first treat the case $q\geq1$. By \eqref{false} and $q\geq1$,
we may assume that $S^q(x,t)\leq S(x,t)$ for all $(x,t)\in\bar\Omega\times[T_1,\infty)$.
So one can see from the first equation in \eqref{density} and \eqref{2.7} that
 $$
  {   \partial_t S}-d_S\Delta S\geq -\theta e^{-\delta(t-T_1)}S,\ \ \ x\in\Omega,\,t\geq T_1,
 $$
where $\theta=\sigma^0\Big(\max_{x\in\bar\Omega}I(x,T_1)\Big)^p$ and $\delta=p\sigma_0/2$.
By considering the ODE problem
 \begin{equation}
 \left \{\begin{array}{lll}
 \medskip
 \dis  {   w'}=-\theta e^{-\delta(t-T_1)}w,\ t>T_1,\\
 \dis w(T_1)=\max_{x\in\bar\Omega}S(x,T_1)>0,
 \end{array}\right.
 \label{auxi-3}
 \end{equation}
we have
 $$
 S(x,t)\geq w(t),\ \ \ \forall x\in\bar\Omega,\,t\geq T_1.
 $$
On the other hand, solving \eqref{auxi-3} yields
 $$
 w(t)=\max_{x\in\bar\Omega}S(x,T_1)e^{\frac{\theta}{\delta}[e^{-\delta(t-T_1)}-1]}
 \geq\max_{x\in\bar\Omega}S(x,T_1)e^{-\frac{\theta}{\delta}},\ \ \ t\geq T_1.
 $$
Henceforth, we have
 $$
 S(x,t)\geq \max_{x\in\bar\Omega}S(x,T_1)e^{-\frac{\theta}{\delta}},\ \ \ \forall x\in\bar\Omega,\,t\geq T_1,
 $$
which leads to a contradiction.

We now assume that $0<q<1$ and $\gamma\geq\sigma_0$ on $\bar\Omega\times[0,\infty)$.
By \eqref{false} and $p\geq1$,
one may assume that $$I^p(x,t)\leq I(x,t),\ \ \gamma(x,t)-\beta(x,t)S^q>0$$ for all $(x,t)\in\bar\Omega\times[T_1,\infty)$.
So it is clear from the first equation in \eqref{density} that
 $$
 {   \partial_t S}-d_S\Delta S>0,\ \ \ x\in\Omega,\,t\geq T_1,
 $$
from which one easily knows that $S(x,t)\geq\min_{\bar\Omega}S(x,T_1)>0$ for all $(x,t)\in\bar\Omega\times[T_1,\infty)$, a contradiction with \eqref{false}. Thus, we have proved $S_*>0$.

Finally, we show $S^q_*\leq\sup_{\Omega\times(0,\infty)}\frac{\gamma+\mu}{\beta}$ for $p=1$.
Indeed, if this is false, then there exists $T_2>0$ such that $(\beta S^q-\gamma-\mu)I>0$ on $\bar\Omega\times [T_2, \infty)$. By the equation of $I$, we have
$$
 {   \partial_t I}-d_I\Delta I>0, \ \ x\in\Omega, t\ge T_2,
$$
from which we obtain $I(x, t)\ge \min_{x\in\bar\Omega} I(x, T_2)>0$. This contradicts the fact that $I(\cdot, t)\rightarrow 0$ in $C(\bar\Omega)$ as $t\rightarrow\infty$. The proof is complete.
\end{proof}

\section{Proof of Theorem \ref{th-a4}}
In this section, we consider \eqref{density} with $\mu=0$; in this case, \eqref{density} becomes the following SIS 
 model:
 \begin{equation}
 \left\{ \begin{array}{llll}
 {   \partial_t S}-d_S\Delta S=-\beta(x,t)S^qI^p+\gamma(x,t)I,&x\in\Omega,\,t>0, \medskip\\
 {   \partial_t I}-d_I\Delta I=\beta(x,t)S^qI^p-\gamma(x,t)I,&x\in\Omega,\,t>0, \medskip\\
  {   \partial_\nu S}= {   \partial_\nu I}=0,&x\in\partial\Omega,\,t>0, \medskip\\
 \dis S(x,0)=S_0(x),\,I(x,0)= I_0(x),&x\in\Omega.
 \end{array}\right.
 \label{SIS}
 \end{equation}
 Adding up the first two equations and integrating over $\Omega\times (0, t)$, we obtain
 $$
 \int_\Omega (S(x, t)+I(x, t)) dx=\int_\Omega (S_0+I_0) dx.
 $$
This leads us to assume that the total population is a constant, i.e.,
 \begin{equation}\label{initialC}
 \int_\Omega (S_0+I_0)=N,
 \end{equation}
for a fixed constant $N>0$.


Let $X=C(\bar\Omega)$ and $C_\omega$ be the space of all periodic continuous functions from $\mathbb{R}$ to $X$ with period $\omega$.

When $p=1$ in system \eqref{SIS}, it is not hard to check that $E_0:=({N}/{|\Omega|}, 0)$
is the unique disease-free equilibrium of \eqref{SIS}.
Linearizing the second equation of \eqref{SIS} at $E_0$, we get
  \begin{equation}
  {   \partial_t I}-d_I\Delta I=\beta(x, t)\left(\frac{N}{|\Omega|}\right)^q I-\gamma(x, t) I.
 \end{equation}
Let $V(t, s)$ be the evolution operator on $C_\omega$ induced by the solution of
  \begin{equation}
  \left\{
  \begin{array}{lll}
 {   \partial_t I}-d_I\Delta I=-\gamma(x, t) I, &x\in\Omega,\,t>0, \medskip\\
 {   \partial_\nu I}=0, &x\in\partial\Omega,\,t>0, \medskip\\
   \dis I(x,0)= I_0(x), &x\in\Omega.
   \end{array}
   \right.
 \end{equation}
Let $\mathcal{L}: C_\omega\rightarrow C_\omega$ be given by
$$
\mathcal{L}(\phi)(t):= \int_{-\infty}^t V(t, s)\beta\left(\frac{N}{|\Omega|}\right)^q\phi(\cdot, s)ds, \ \phi\in C_\omega.
$$
Then the basic reproduction number $\mathcal{R}_0$ is defined as the spectral radius of $\mathcal{L}$, i.e.,
\begin{equation}\label{Rdef}
\mathcal{R}_0=r(\mathcal{L}).
\end{equation}
Similar to \cite{PZ}, $1-\mathcal{R}_0$ has the same sign as $\lambda_0$, where $\lambda_0$ is the principal eigenvalue of the  periodic-parabolic eigenvalue problem
  \begin{equation}\label{eig}
  \left\{
  \begin{array}{lll}
  {   \partial_t \varphi}-d_I\Delta \varphi=\beta(x, t)\left(\frac{N}{|\Omega|}\right)^q\varphi-\gamma(x, t)\varphi+\lambda\varphi, &x\in\Omega,\,t>0, \medskip \\
  {   \partial_\nu \varphi}=0, &x\in\partial\Omega,\,t>0, \medskip \\
   \dis \varphi(x,\omega)= \varphi(x, 0), &x\in\Omega.
   \end{array}
   \right.
 \end{equation}


Before proving Theorem \ref{th-a4}, we prepare the following two results on the uniform weak persistence property. Since $S^qI^p$ is not locally Lipschitz unless $S, I>0$,  the solutions of \eqref{density1} do not induce a semiflow on a complete metric space. Therefore, we cannot follow the standard arguments in dynamical system theory here. Our proof of uniform weak persistence is inspired by \cite{Ducrot}.

\begin{Lemma}[Uniform weak persistence] \label{lemma_weak}
Assume that {\rm(A1)-{\color{red}(A2)}} and {\color{red}\rm(A4)} hold, and let $0<p<1$. Then there exists $\epsilon_0>0$ independent of initial data such that for any solution $(S, I)$ of \eqref{SIS}-\eqref{initialC} we have
$$
\limsup_{t\rightarrow\infty} \min_{x\in\bar\Omega} S(x, t),\ \  \limsup_{t\rightarrow\infty}  \min_{x\in\bar\Omega} I(x, t) \ge \epsilon_0.
$$
\end{Lemma}

\begin{proof}
Suppose on the contrary that the conclusion does not hold. Then there exist solutions $\{(S_n(x, t), I_n(x, t))\}$ of \eqref{SIS}-\eqref{initialC}  such that one of the following two cases happens:

 \begin{enumerate}
 \item[] \hspace{-0.5cm}{\bf Case 1.}  $\min_{x\in\bar\Omega} S_n(x, t+t_n)<1/n$ for all $t\ge 0$. In particular, $S_n(x_n, t_n)<1/n$;

 \item[] \hspace{-0.5cm}{\bf Case 2.} $\min_{x\in\bar\Omega} I_n(x, t+t_n)<1/n$ for all $t\ge 0$. In particular, $I_n(x_n, t_n)<1/n$;
 \end{enumerate}
where $x_n\in\bar\Omega$ for all $n$ and $t_n\rightarrow \infty$. Restricted to a subsequence if necessary, we may assume $x_n\rightarrow x_0\in\bar\Omega$, $\beta(x, t+t_n)\rightarrow \hat\beta(x, t)$ and $\gamma(x, t+t_n)\rightarrow\hat\gamma(x, t)$ uniformly for $x\in\bar\Omega$ and $t\in \mathbb{R}$.

Define
$$
(S^n(x, t), I^n(x, t))=(S_n(x, t+t_n), I_n(x, t+t_n)), \ \ (x, t)\in\bar\Omega\times [-t_n, \infty).
$$
Then
$$
S^n(x_n, 0)<{   1/n}\ \  \text{or} \ \  I^n(x_n, 0)<{   1/n}.
$$

In view of Theorem \ref{Theorem_boundaa}, $\{(S^n, I^n)\}$ is uniformly bounded in $C_{loc}(\bar\Omega\times\mathbb{R})\times C_{loc}(\bar\Omega\times\mathbb{R})$. By the parabolic-type $L^p$ estimate, $\{(S^n, I^n)\}$ is uniformly bounded in $W^{2, 1; p}_{loc}(\bar\Omega\times\mathbb{R})\times W^{2, 1; p}_{loc}(\bar\Omega\times\mathbb{R})$. Therefore, from the Sobolev embedding theorem, up to a subsequence if necessary, it follows that
$$
S^n(x, t)\rightarrow S^\infty(x, t) \text{ and } I^n(x, t)\rightarrow I^\infty(x, t)   \text{ in } C_{loc}(\bar\Omega\times\mathbb{R}),
$$
where $(S^\infty, I^\infty)$ is a bounded nonnegative entire solution of the following problem:
 \begin{equation}
 \left\{ \begin{array}{llll}
{   \partial_t S}-d_S\Delta S=-\hat\beta(x,t)S^qI^p+\hat\gamma(x,t)I,&x\in\Omega,\,t\in\mathbb{R}, \medskip\\
{   \partial_t I}-d_I\Delta I=\hat\beta(x,t)S^qI^p-\hat\gamma(x,t)I,&x\in\Omega,\,t\in\mathbb{R},\medskip\\
 {   \partial_\nu S}={   \partial_\nu I}=0,&x\in\partial\Omega,\,t\in\mathbb{R}, \medskip\\
  \int_\Omega (S(x, t)+I(x, t))dx=N, &t\in\mathbb{R}.
 \end{array}\right.
 \label{SIS_entire}
 \end{equation}
 By the second equation of \eqref{SIS_entire},
 $$
{   \partial_t I^\infty}-d_I\Delta I^\infty\ge -\hat\gamma(x,t)I^\infty, \ \ (x, t)\in \bar\Omega\times\mathbb{R}.
 $$
 It follows from the comparison principle that either $I^\infty=0$ or $I^\infty(x, t)>0$ for all $(x, t)\in \bar\Omega\times\mathbb{R}$. For the latter case, using the comparison principle and Hopf Lemma as in the proof of Lemma \ref{lemma_local}, one can show that $S^\infty(x, t)>0$ for all $(x, t)\in \bar\Omega\times\mathbb{R}$. Therefore, $(S^\infty, I^\infty)$ satisfies exactly one of the following two possibilities:
 
 (I). $(S^\infty, I^\infty)=(S^\infty, 0)$, where $S^\infty$ is a bounded nonnegative entire solution of the following problem:
  \begin{equation}
 \label{SIS_entire1}
 \left\{ \begin{array}{llll}
{   \partial_t S}-d_S\Delta S=0,&x\in\Omega,\,t\in\mathbb{R}, \medskip\\
 
{   \partial_\nu S}=0,&x\in\partial\Omega,\,t\in\mathbb{R},\medskip\\
 \int_\Omega S(x, t)dx=N, &t\in\mathbb{R}.
 \end{array}\right.
 \end{equation}
  By the maximum principle, $S^\infty(x, t)>0$ for all $(x, t)\in \bar\Omega\times\mathbb{R}$ (Indeed, it is not hard to show $S^\infty=N/|\Omega|$).

 (II). $S^\infty(x, t), I^\infty(x, t)>0$ for all  $(x, t)\in \bar\Omega\times\mathbb{R}$.

\vskip6pt
We now show that either case will lead to a contradiction.

In \textbf{Case 1},  since $S^n(x_n, 0)<1/n$, we then have $S^\infty(x_0, 0)=0$, which contradicts with (I) or (II) above.

In \textbf{Case 2},  since $I^n(x_n, 0)<1/n$, it is necessary that $I^\infty(x_0, 0)=0$.
Therefore, from the above analysis, we have $(S^\infty, I^\infty)=(S^\infty, 0)$ satisfying (I).

 We claim that, for any $h>0$, the following hold:
\begin{equation}\label{ine1}
 \limsup_{n\rightarrow\infty} \inf_{(x, t)\in\bar\Omega\times [-h, \infty)} S^n(x, t)>0
\end{equation}
and
\begin{equation}\label{ine2}
 \lim_{n\rightarrow\infty} \sup_{(x, t)\in\bar\Omega\times [-h, \infty)} I^n(x, t)=0.
\end{equation}

We prove these two claims by contradiction. Suppose on the contrary that \eqref{ine1} does not hold. Then there exist $h>0$  and a subsequence of $\{S^n\}$, still denoted by itself, such that
$$
S^n(y_n, \tau_n)<{   1/n}
$$
for some $\tau_n\geq-h$ and $y_n\in\bar\Omega$. As before, we may assume $y_n\rightarrow y_0\in\bar\Omega$, $\beta(x, t+t_n+\tau_n)\rightarrow \bar\beta(x, t)$ and $\gamma(x, t+t_n+\tau_n)\rightarrow \bar\gamma(x, t)$ uniformly in $x$ and $t$. Define
$$
(\bar S^n(x, t), \bar I^n(x, t))=(S^n(x, t+\tau_n), I^n(x, t+\tau_n)), \ \ x\in\bar\Omega, t>-\tau_n-t_n.
$$
Then
\begin{equation}\label{barsn}
\bar S^n(y_n, 0)<{   1/n}
\end{equation}
and, up to a subsequence if necessary, 
$$(\bar S^n, \bar I^n)\rightarrow (\bar S^\infty, \bar I^\infty) \ \ \text{in}\  C_{loc}(\bar\Omega\times\mathbb{R})\times C_{loc}(\bar\Omega\times\mathbb{R}),
$$
where $(\bar S^\infty, \bar I^\infty)$ is a nonnegative bounded entire solution of \eqref{SIS_entire} with $\hat \beta$ and $\hat\gamma$ replaced by $\bar\beta$ and $\bar\gamma$, respectively.
By \eqref{barsn}, we have $\bar S^\infty(y_0, 0)=0$.
This contradicts with $\bar S^\infty(x, t)>0$ for all $(x, t)\in \bar\Omega\times\mathbb{R}$. Thus, \eqref{ine1} is verified.

Suppose on the contrary that \eqref{ine2} does not hold. Then there exist $h, \epsilon_0>0$ and a subsequence of $\{I^n\}$, still denoted by itself,  such that
$$
I^n(z_n, s_n)>\epsilon_0
$$
for some $s_n\geq-h$ and $z_n\in\bar\Omega$. We may assume $z_n\rightarrow z_0\in\bar\Omega$, $\beta(x, t+t_n+s_n)\rightarrow \tilde\beta(x, t)$ and $\gamma(x, t+t_n+s_n)\rightarrow \tilde\gamma(x, t)$ uniformly in $x$ and $t$.
Define
$$
(\tilde S^n(x, t), \tilde I^n(x, t))=(S^n(x, t+s_n), I^n(x, t+s_n)), \ \ x\in\bar\Omega, t>-s_n-t_n.
$$
Then, we have
\begin{equation}\label{tildeI1}
\tilde I^n(z_n, 0)>\epsilon_0
\end{equation}
 and
\begin{equation}\label{tildeI2}
 \min_{x\in\bar\Omega} \tilde I^n(x, t)<{   1/n}, \  \forall t>-s_n
\end{equation}
  As before, up to a subsequence if necessary, 
  $$
  (\tilde S^n, \tilde I^n)\rightarrow (\tilde S^\infty, \tilde I^\infty) \ \ \text{in}\  C_{loc}(\bar\Omega\times\mathbb{R})\times C_{loc}(\bar\Omega\times\mathbb{R}),
$$
where $(\tilde S^\infty, \tilde I^\infty)$ is a nonnegative bounded entire solution of \eqref{SIS_entire} with $\hat \beta$ and $\hat\gamma$ replaced by $\tilde\beta$ and $\tilde\gamma$, respectively.
By \eqref{tildeI1}, $\tilde I^\infty(z_0, 0)>\epsilon_0$. By \eqref{tildeI2}, we must have $\tilde I^\infty=0$, which is a contradiction. This proves \eqref{ine2}.

By means of \eqref{ine1}-\eqref{ine2} and $0<p<1$, there exist $\delta_0>0$ and $N_0>0$ such that
$$
\beta(x, t+t_{N_0})\frac{[S^{N_0}(x, t)]^q}{[I^{N_0}(x, t)]^{1-p}}-\gamma(x, t+t_{N_0})>\delta_0\ \ \text{for all } (x, t)\in\bar\Omega\times [0, \infty).
$$
Therefore, $I^{N_0}(x, t)$ satisfies
$$
{   \partial_t I^{N_0}}-d_I\Delta I^{N_0} \ge \delta_0 I^{N_0}.
$$
By the parabolic comparison principle, we have
$$
I^{N_0}(x, t)\ge e^{\delta_0 t} I^{N_0}(x, 0)\rightarrow \infty,
$$
which is impossible. This finishes the proof in {\bf Case 2}.
\end{proof}

\begin{Lemma}[Uniform weak persistence] \label{lemma_weak2}
Assume that {\rm(A1)-{\color{red}(A2)}} and {\color{red}{\rm(A4)}} hold, and let $p=1$. If $\mathcal{R}_0>1$, then there exists $\epsilon_0>0$ independent of initial data such that for any solution $(S, I)$ of \eqref{SIS}-\eqref{initialC} we have
$$
\limsup_{t\rightarrow\infty} \min_{x\in\bar\Omega} S(x, t),\ \  \limsup_{t\rightarrow\infty}  \min_{x\in\bar\Omega} I(x, t) \ge \epsilon_0.
$$
\end{Lemma}
\begin{proof} We juts need to modify the proof of Lemma \ref{lemma_weak}. Let $S^n,\,I^n$ be as in proof of Lemma \ref{lemma_weak}, and we can obtain \eqref{ine1}-\eqref{ine2} using the same argument as there. 

We further claim that for any $h>0$:
\begin{equation}\label{ine3}
 \lim_{n\rightarrow\infty} \sup_{(x, t)\in\bar\Omega\times [-h, \infty)} \left |S^n(x, t)-\frac{N}{|\Omega|} \right |=0.
\end{equation}
Suppose on the contrary that \eqref{ine3} does not hold. Then there exist $h,\epsilon_0>0$ such that
$$
 \left |S^n(z_n, s_n)-\frac{N}{|\Omega|} \right |>\epsilon_0
$$
for some $s_n\geq-h$ and $z_n\in\bar\Omega$. We may assume $z_n\rightarrow z_0\in\bar\Omega$, $\beta(x, t+t_n+s_n)\rightarrow \tilde\beta(x, t)$ and $\gamma(x, t+t_n+s_n)\rightarrow \tilde\gamma(x, t)$ uniformly in $x$ and $t$.

Similar to the proof of Lemma \ref{lemma_weak}, define
$$
(\tilde S^n(x, t), \tilde I^n(x, t))=(S^n(x, t+s_n), I^n(x, t+s_n)), \ \ x\in\bar\Omega, t>-s_n-t_n.
$$
Then,
\begin{equation}\label{tildeSS}
 \left |\tilde S^n(z_n, 0)-\frac{N}{|\Omega|} \right |>\epsilon_0.
\end{equation}
Moreover, up to a subsequence if necessary, 
$$
(\tilde S^n, \tilde I^n)\rightarrow (\tilde S^\infty, \tilde I^\infty) \ \ \text{in}\  C_{loc}(\bar\Omega\times\mathbb{R})\times C_{loc}(\bar\Omega\times\mathbb{R}),
$$
where $(\tilde S^\infty, \tilde I^\infty)$ is a nonnegative bounded entire solution of \eqref{SIS_entire} with $\hat \beta$ and $\hat\gamma$ replaced by $\tilde\beta$ and $\tilde\gamma$, respectively. By \eqref{ine2}, we must have $\tilde I^\infty=0$. Therefore, $\tilde S^\infty$ is a nonnegative bounded entire solution of \eqref{SIS_entire1}. 

We further conclude that $\tilde S^\infty= {N}/ {|\Omega|}$. To see this, we can write $\tilde S^\infty(x,t)=\sum_{i=1}^\infty a_i(t)\phi_i(x)$, where $\{\phi_i\}_{i=1}^\infty$ are the eigenvectors of $-\Delta$ with homogeneous Neumann boundary condition, and they are an orthonormal basis of $L^2(\Omega)$. Let $\{\lambda_i\}_{i=1}^\infty$ with $0=\lambda_1<\lambda_2\leq\lambda_3\cdots$ be the corresponding eigenvalues. Clearly, $\phi_1$ is constant.  

Substituting  $\tilde S^\infty(x,t)=\sum_{i=1}^\infty a_i(t)\phi_i(x)$ into the first equation of \eqref{SIS_entire1}, we can easily see that 
\begin{equation}\label{tildeSS-aa}
\tilde S^\infty(x, t)=\sum_{i=1}^\infty   a_i e^{\lambda_i t} \phi_i(x),
\end{equation}
where $a_i\,(i\geq1)$ are constants. Multiplying \eqref{tildeSS-aa} by $\phi_i(x)$ for any given $i\geq1$, and then integrating over $\Omega$, we deduce that
$$
a_i e^{\lambda_i t} =\int_\Omega  \tilde S^\infty(x, t)\phi_i(x)dx, \ \ \forall i\ge 1.
$$
 Thanks to the boundedness of $\tilde S^\infty$ and the fact $\lambda_i>0$ for all $i\geq2$, it is easily seen that $a_i=0$ for all $i\ge 2$. Recall that $\lambda_1=0$ and $\phi_1$ is constant. It then follows that $\tilde S^\infty$ is constant. By the third equation of \eqref{SIS_entire1}, we have $\tilde S^\infty= {N}/ {|\Omega|}$.

In light of \eqref{tildeSS}, it is necessary that
\begin{equation*}
 \left |\tilde S^\infty(z_0, 0)-\frac{N}{|\Omega|} \right |>\epsilon_0.
\end{equation*}
This contradicts with $\tilde S^\infty= {N}/ {|\Omega|}$, and \eqref{ine3} is thus proved.

Note that $\mathcal{R}_0>1$ and $1-\mathcal{R}_0$ has the same sign with the principal eigenvalue of problem \eqref{eig}. Then, we can choose $\epsilon_0>0$ small enough so that the principal eigenvalue $\lambda_{\epsilon_0}$ of the following problem is negative:
  \begin{equation}
  \left\{
  \begin{array}{lll}
{   \partial_t \varphi}-d_I\Delta \varphi=\beta(x, t)\left(\frac{N}{|\Omega|} - \epsilon_0 \right)^q\varphi-\gamma(x, t)\varphi+\lambda\varphi, &x\in\Omega,\,t>0,  \medskip\\
 {   \partial_\nu \varphi}=0, &x\in\partial\Omega,\,t>0,  \medskip \\
   \dis \varphi(x,\omega)= \varphi(x, 0), &x\in\Omega.
   \end{array}
   \right.
 \end{equation}
 Let $\varphi_{\epsilon_0}>0$ be a corresponding eigenvector of $\lambda_{\epsilon_0}$.  By \eqref{ine3}, there exists $N_0>0$ such that
$$
\beta [S^{N_0}(x, t)]^q-\gamma \ge \beta \left(\frac{N}{|\Omega|} - \epsilon_0 \right)^q-\gamma, \ \ \forall (x, t)\in\bar\Omega\times [0, \infty).
$$
Therefore, $I^{N_0}(x, t)$ satisfies
 \begin{equation*}
 \left\{ \begin{array}{llll}
 \medskip
{   \partial_t  I^{N_0}}-d_I\Delta I^{N_0} \ge  I^{N_0}\left(\beta \left(\frac{N}{|\Omega|} - \epsilon_0 \right)^q-\gamma\right),&x\in\Omega,\,t\in\mathbb{R},  \\
{   \partial_\nu I^{N_0}}=0,&x\in\partial\Omega,\,t\in\mathbb{R}, \medskip\\
I^{N_0}(x, 0)\ge \delta \varphi_{\epsilon_0}(x, 0), &x\in\Omega,
  \end{array}\right.
 \end{equation*}
where $\delta>0$ is small. Then $I^{N_0}$ is a supersolution of the following problem:
 \begin{equation}\label{IN1}
 \left\{ \begin{array}{llll}
 \medskip
{   \partial_t u}-d_I\Delta u =  u\left(\beta \left(\frac{N}{|\Omega|} - \epsilon_0 \right)^q-\gamma\right),&x\in\Omega,\,t\in\mathbb{R},\\
{   \partial_\nu u}=0,&x\in\partial\Omega,\,t\in\mathbb{R},  \medskip\\
u(x, 0)= \delta \varphi_{\epsilon_0}(x, 0), &x\in\Omega.
  \end{array}\right.
 \end{equation}
It is not hard to check that $u=\delta e^{-\lambda_{\epsilon_0}t}\varphi_{\epsilon_0}(x, t)$ is the unique solution of \eqref{IN1}. By the comparison principle and $\lambda_{\epsilon_0}<0$, we have
$$
I^{N_0}(x, t)\ge \delta e^{-\lambda_{\epsilon_0}t}\varphi_{\epsilon_0}(x, t) \rightarrow \infty,\ \ \text{as } t\rightarrow\infty,
$$
which contradicts the boundedness of $I^{N_0}$. This completes the proof.
\end{proof}

\vskip8pt
Now we are ready to prove the uniform persistence of the solutions and the existence
of a positive $\omega$-periodic  solution of \eqref{SIS}-\eqref{initialC}. The definitions and results from  dynamical systems theory used below  can be found in the appendix. Since the semiflow induced by the solutions of \eqref{SIS}-\eqref{initialC} is not defined on a complete metric space, small modifications are necessary.
\\
\begin{proof}[Proof of Theorem \ref{th-a4}] Let $X^+$  be the positive cone of $X$.
Let $A$ be the complete metric space given by
$$
A=\left\{(u, v)\in X^+\times X^+: \ \int_\Omega (u+v) dx=N\right\}
$$
with distance induced by the norm of $X\times X$.

Let $\rho: A\rightarrow [0, \infty)$ be given by
$$
\rho((u, v))=\min\{\min_{x\in\bar\Omega} u(x),\ \min_{x\in\bar\Omega} v(x) \}, \ (u, v)\in A.
$$
We also set $A=A_0\cup \partial A_0$, where
$$
A_0=\{(u, v)\in A:\ u(x),\ v(x)> 0 \text{ for all } x\in\bar\Omega\}
$$
and
$$
\partial A_0=\{(u, v)\in A:\ u(x)=0 \ \text{or } v(x)= 0 \text{ for some } x\in\bar\Omega\},
$$
with $A_0=\rho^{-1}((0, \infty))$ and $\partial A_0=\rho^{-1}(\{0\})$. It is not hard to check that $\partial A_0$ is relatively closed and $A_0$ is relatively open. (Since the nonlinear term $S^qI^p$ may prevent the solution of \eqref{SIS}-\eqref{initialC} from being unique if $(S_0, I_0)\in \partial A_0$, we work in $A_0$). 

Let $T(t): A_0\rightarrow A_0$ be the $\omega$-periodic semiflow induced by the solutions of \eqref{SIS}-\eqref{initialC}, i.e., $T(t)(S_0, I_0)=(S(\cdot, t), I(\cdot, t))$, $t\ge 0$, which satisfies $T(t+\omega)=T(t)$ for all $t\ge 0$. Let $\mathcal{P}=T(\omega): A_0\rightarrow A_0$ be the Poinc\'are map of \eqref{SIS}.

By Theorem \ref{th-a}, $\mathcal{P}: A_0\rightarrow A_0$ is point dissipative.
Moreover, $\mathcal{P}$ is compact because of the dissipation terms in \eqref{SIS} (One can actually see this from the semigroup computation in the proof of Theorem \ref{Theorem_bound0}: Firstly, the uniform $L^\infty$-bound in Theorem \ref{th-a} (i.e., $M_\infty$) depends only on the $L^\infty$-norm of the initial data, and so $\mathcal{P}$ maps bounded sets into bounded sets in $A$; then, in \eqref{semig}, one can see that the constant $C$ depends only on the  $L^\infty$-norm of the initial data; finally, by the compactness of the embedding $X_\alpha\subset C(\bar\Omega)$, $\mathcal{P}$ maps bounded sets into precompact  sets in $A$).

Furthermore, due to Lemmas \ref{lemma_weak}-\ref{lemma_weak2}, $\mathcal{P}$ is weakly $\rho$-uniformly persistent, i.e., there exists $\epsilon_0'>0$ such that
$$
\limsup_{n\rightarrow\infty} \rho(\mathcal{P}^n(S_0, I_0))>\epsilon_0', \ \ \forall (S_0, I_0)\in A_0.
$$
Applying Propositions \ref{thm_comp}-\ref{thm_persistent}, $\mathcal{P}$ is $\rho$-uniformly persistent, i.e., there exists $\epsilon''_0>0$ such that
$$
\liminf_{n\rightarrow\infty} \rho(\mathcal{P}^n(S_0, I_0))>\epsilon''_0, \ \ \forall (S_0, I_0)\in A_0,
$$
which implies \eqref{persist}.

Finally, we can apply Proposition \ref{thm_fix} to prove the existence of a positive periodic solution.  It suffices to show that $\mathcal{P}$ maps $\rho$-strongly bounded subsets of $A_0$ to   $\rho$-strongly bounded subsets of $A_0$. Let $B$ be a $\rho$-strongly bounded subset of $A_0$, i.e., $B$ is bounded and there exists $\epsilon >0$ such that
$$
\min_{x\in\bar\Omega} S_0(x), \ \min_{x\in\bar\Omega} I_0(x) \ge \epsilon, \ \ \forall (S_0, I_0)\in B.
$$
It is not hard to see that $\mathcal{P}(B)$ is bounded by Theorem \ref{th-a}.   Let $(S_0, I_0)\in B$ and $(S(x, t), I(x, t))$ be the solution of \eqref{SIS}-\eqref{initialC}.
Then $\mathcal{P}(S_0, I_0)=T(\omega)(S_0, I_0)=(S(\cdot, \omega), I(\cdot, \omega))$.  By the second equation of \eqref{SIS}, we have
$$
{   \partial_t I}-d_I\Delta I>-\sigma^0 I, \ \ x\in\Omega, t\ge 0.
$$
Using the comparison principle, we have
\begin{equation}\label{lbound1}
I(x, t) \ge e^{-\sigma^0 t} I_0(x)\ge \epsilon e^{-\sigma^0 t}, \ \ x\in \bar\Omega.
\end{equation}
From the first equation of \eqref{SIS} it follows that
$$
{   \partial_t S}-d_S\Delta S\ge -\sigma^0 S^q C^p+ \gamma_m(t) \epsilon e^{-\sigma^0 t}, \  \  x\in\Omega,\,t\in (0, \omega],\\
$$
where
$
\gamma_m(t)= \min_{x\in\bar\Omega} \gamma(x, t)
$
and
$C$ is chosen such that $I(x, t)\le C$ for all $x\in\bar\Omega$ and $t\in [0, \omega]$ uniformly for $(S_0, I_0)\in B$. Therefore, the comparison principle infers that
\begin{equation}\label{lbound2}
S(x, t)\ge w(t), \ \ \forall x\in\bar\Omega, t\in [0, \omega],
\end{equation}
where $w(t)$ is the solution of the following ordinary differential equation:
 \begin{equation*}
 \left \{\begin{array}{lll}
 \medskip
 {   w'}= -\sigma^0  C^p w^q + \epsilon \gamma_m(t)  e^{-\sigma^0 t},\ t\in (0, \omega],\\
 \dis w(0)=\epsilon.
 \end{array}\right.
 \end{equation*}
Combining \eqref{lbound1}-\eqref{lbound2}, we have $S(x, \omega)\ge w(\omega)$ and $I(x, \omega)\ge \epsilon e^{-\sigma^0 \omega}$ for all $x\in \bar\Omega$. This indicates that $\mathcal{P}(B)$ is a $\rho$-strongly bounded subset of $A_0$.  Therefore, by Proposition \ref{thm_fix}, $\mathcal{P}$ has at least one fixed point in $A_0$, equivalently, \eqref{SIS} has at least one positive $\omega$-periodic solution.
\end{proof}

\begin{Remark}
 \begin{itemize}
 \item[\rm{(i)}] If $p=1, q\ge 1$ and $\mathcal{R}_0>1$, with the aid of Theorem \ref{th-a}, we just need to slightly modify the analysis of \cite[Theorem 3.3]{PZ}  to prove the uniform persistence result (i.e., we do not need to prove the uniform weak persistence first).
     
  \item[\rm{(ii)}]  When $p=1$, the same argument of \cite[Theorem 2.1]{DW} allows one to conclude that the solution of
\eqref{SIS} is bounded by a positive constant which depends on the initial data; however, such estimates are insufficient for us
to obtain Theorem \ref{th-a4}.  
 \end{itemize}     
\end{Remark}

\section{Discussion}

In this section, we first discuss further applications of our analysis used in this paper to some other epidemic models.
Then we interpret the biological implications of our results and conclude the influence of
the parameters and coefficients of the model on the dynamical behavior of disease transmissions.

\subsection{Other related epidemic models} We want to mention that the mathematical techniques developed in the previous sections can  be carried over to other types of infection incidence functions, including the following ones:
 \begin{itemize}
 \item[\rm{(i)}] The binomial incidence function $S\ln(1+k I)$ with $k\ge 0$ (\cite{Ba,CG,McCallum});

 \item[\rm{(ii)}]  The incidence function $\frac{S^qI^p}{1+I^\ell}$ with constants $p,q>0,\,\ell\geq0$ (\cite{DV1,HV,LHL,LR});

 \item[\rm{(iii)}]  The media effect incidence function $\frac{S^qI^p}{1+I^\ell}e^{-I}$ with constants $p,q>0,\,\ell\geq0$ (\cite{CSZ,CTZ}).

 \end{itemize}

More precisely, if $S^qI^p$ is replaced by $S\ln(1+k I)$ in  \eqref{density}, we can
prove the same results as in Theorems \ref{th-a}, \ref{result1} and \ref{th-a4} for $p=1$;
if $S^qI^p$ is replaced by $\frac{S^qI^p}{1+I^\ell}$ or $\frac{S^qI^p}{1+I^\ell}e^{-I}$ in \eqref{density}, we can
prove the same results as in Theorems \ref{th-a}, \ref{result1} and \ref{th-a4}.

\subsection{Conclusion} Usually, a priori  $L^\infty$-bounds are the starting point to study the long-time behavior of the solutions of a reaction-diffusion system. In this paper, we establish the $L^\infty$-bounds for \eqref{density1} first. Our results  include a range of parameters which are not covered by \cite{Fitz,Morgan2, Pierre}, and the technique developed in the proof of Theorem \ref{Theorem_boundaa} under (H2)  may find further applications in other reaction-diffusion systems. We remark that if $p, q\ge 1$ one may apply the results in \cite{Fitz,Morgan2} to obtain the $L^\infty$-bounds in Theorem \ref{th-a} for \eqref{density} but not for \eqref{density1} in general; if $p<1$ or $q<1$, the positivity of the solutions is required to ensure the unique extension of the solutions, and therefore the analysis is more subtle.

Based on the $L^\infty$-bounds, we investigate the long-time behavior of the solutions of  \eqref{density} in the following two cases:
 \begin{equation}\nonumber
 \begin{array}{l}
 \mbox{(i)\ \, $\mu>0$ \ \ \ and\ \ \ (ii)\ \, $\mu=0$}.
 \end{array}
 \end{equation}
The global dynamics of \eqref{density} are very different for these two cases.

The case (i) states that there are individuals who die from the disease and thus the total population number of susceptible and infected hosts is decreasing in time $t$ as seen from \eqref{control-mass-d}.  In this case, Theorem \ref{result1} shows that  $I(x, t)$ converges to zero, which means that the infection will become extinct in the long run. However, the long-time behavior of $S(x, t)$ depends on the parameter $p$:
 \begin{itemize}
 \item  When $p\geq1$, the density of susceptible individuals converges to a positive constant, which means that the susceptible population will distribute homogeneously in the whole habitat eventually;

\item When $0<p<1$,  the density of   susceptible individuals converges to zero which means that  the disease is fatal enough so that it drives its hosts to extinction.
 \end{itemize}

The case (ii) biologically means that the disease is not fatal. From \eqref{control-mass-d} it follows that the total population of susceptible and infected remains constant all the time. According to Theorem \ref{th-a4} and Remark \ref{re-p1},
we see that
 \begin{itemize}
 \item When $0<p<1$ or $p=1$ with the basic production number $\mathcal{R}_0>1$, the susceptible and infected populations uniformly persist in the whole habitat in the long run;

 \item When $p=1$ and $\mathcal{R}_0\leq1$, the disease dies out and the susceptible population persists.

 \item When $p>1$, Remark \ref{re-p1}(ii) indicates that the dynamics of \eqref{SIS}-\eqref{initialC} will  depend on the initial data and the  susceptible population will not be driven to extinction.
 \end{itemize}

The above discussion shows that in an SI/SIS system with nonlinear incidence function $S^qI^p\,(p,\,q>0)$, the power $p$ and the disease-induced death rate $\mu$ are vital factors in determining the global dynamics; in particular, if the disease-induced death rate $\mu$ is taken into account, the fatal disease causes its hosts to extinction if and only if $0<p<1$. 
  
It is further worth mentioning that in a very recent  work \cite{FCGT}, Farrell \emph{et al.} studied
a class of SI ODE systems in which the incidence function $S^qI^p$ is one of the main focuses. In particular, they explored the roles of the exponents $p,q$ on the extinction of the susceptible population. One may refer to
\cite{FCGT} and the references therein for more  experimental observations and theoretical analysis regarding the phenomenon of host extinction caused by  infectious diseases.

\section{Acknowledgment}
The authors would like to thank the referee and editor for their comments, which lead to improvements of the presentation of the paper.

\section{Appendix}
\subsection{An SI ODE model}
If $\gamma=0$ and $0<q<1$, the solution of \eqref{density} may fail to remain positive (this is the reason we need assumption {\color{red}(A2)}-(ii)).
To see this,  we consider the following SI epidemic  model:
 \begin{equation}
 \left\{ \begin{array}{llll}
 {   S'}=-\beta S^qI^p,&\,t>0, \medskip\\
{   I'}=\beta S^qI^p-\mu I,&\,t>0, \medskip\\
 \dis S(0)=S_0>0,\,I(0)= I_0>0,
 \end{array}\right.
 \label{density-ODE}
 \end{equation}
where the parameters $\beta,\,\mu, \, p$ and $q$ are positive numbers. Denote by $(S,I)$
the unique solution of \eqref{density-ODE}. Clearly,
$(S,\,I)$ exists for all time $t>0$, and $S(t)\geq0,\,I(t)>0$ for all $t>0$.

From the proof of Theorem \ref{result1}, we have already known:
(1) If $0<p<1$, $(S,I)\to(0,0)$ as $t\to\infty$; (2) If $p, q\ge 1$, $(S,I)\to(S_*,0)$ as $t\to\infty$, where $S_*>0$. In addition, we can state the following result.

\begin{Proposition}\label{ODE-result} Suppose $q\in (0, 1)$. Let $(S,I)$ be
the unique solution of \eqref{density-ODE}. The following assertions hold:

 \begin{itemize}

 \item[\rm{(i)}] If $\mu pS^{1-q}(0)=(1-q)\beta I^p(0)$, then
  $(S,I)\to(0,0)$ as $t\to\infty$, and if $\mu pS^{1-q}(0)$ $<(1-q)\beta I^p(0)$, then
  $I\to0$ as $t\to\infty$ and $S(t)=0,\, \forall t\geq T_*$, for some $0<T_*<\infty$.

 \item[\rm{(ii)}] If $p\geq1$ and $pS_0^{1-q}[\mu-\beta S^{q}(0)I^{p-1}(0)]>(1-q)\beta I^p(0)$, then
  $(S,I)\to(S_*,0)$ as $t\to\infty$, for some positive constant $S_*$.

 \end{itemize}

 \end{Proposition}

\begin{proof} As in the proof of Theorem \ref{result1}, one can see that
 $I\to0$ as $t\to\infty$ and $I(t)\geq I(0)e^{-\mu t}$ for all $t\geq0$.
 Furthermore, it is easily observed that $S$ is strictly deceasing on $[0,\infty)$
 and so $\lim_{t\to\infty}S(t)=S_*\geq0$ for some nonnegative number $S_*$.

Using the fact $I(t)\geq I(0)e^{-\mu t}$, $t\geq0$, it follows from the first equation in
\eqref{density-ODE} that
 $$
{   S'}\leq-\beta I^p(0)e^{-p\mu t} S^q,\ \ t>0.
 $$
Hence, $S(t)\leq Z(t)$ for all $t\geq0$, where $Z(t)$ is the unique solution of the problem
 \begin{equation}
 {   Z'}=-\beta I^p(0)e^{-p\mu t} Z^q,\,t>0; \ \ \  Z(0)=S_0>0.
 \label{density-ODE-1}
 \end{equation}
Solving \eqref{density-ODE-1} yields
 $$
 Z^{1-q}(t)=S_0^{1-q}+\frac{(1-q)\beta I^p(0)}{p\mu}[e^{-p\mu t}-1],\ \ t>0.
 $$
This, together with $S(t)\leq Z(t)$ for all $t\geq0$,  implies that $S(t)\to0$ if $S_0^{1-q}-\frac{(1-q)\beta I^p(0)}{p\mu}=0$
and if $S_0^{1-q}-\frac{(1-q)\beta I^p(0)}{p\mu}<0$, then $S(t)=0$ for all $t\geq T_*$, where $T_*\leq T^*$ and $T^*$ is the unique root of
$S_0^{1-q}+\frac{(1-q)\beta I^p(0)}{p\mu}[e^{-p\mu t}-1]=0$.

Next, we verify (ii). Since $\beta S^{q}(0)I^{p-1}(0)<\mu$ and $S$ is decreasing, it easily follows from the equation of $I$ that $I$ is also decreasing on $[0,\infty)$. In particular, we have
$\frac{d I}{d t}\leq[\beta S^{q}(0)I^{p-1}(0)-\mu]I$ for all  $t>0$. This gives $I(t)\leq I(0)e^{-[\mu-\beta S^{q}(0)I^{p-1}(0)]t}$ for all $t>0$. In turn, we get from the equation of $S$ that
 $$
{   S'}\geq -\beta I^p(0)e^{-p[\mu-\beta S^{q}(0)I^{p-1}(0)]t}S^q,\,\ \ \forall t>0.
 $$
Arguing similarly as before, we find that
 $$
 \displaystyle S^q(t)\geq S_0^{1-q}+\frac{(1-q)\beta I^p(0)}{p[\mu-\beta S^{q}(0)I^{p-1}(0)]}
 \left\{e^{-p[\mu-\beta S^{q}(0)I^{p-1}(0)] t}-1\right\},\ \ t>0.
 $$
Therefore,
 $$
 S^q(t)\to S_*^q\geq S_0^{1-q}-\frac{(1-q)\beta I^p(0)}{p[\mu-\beta S^{q}(0)I^{p-1}(0)]}>0,\ \ \mbox{as}\ t\to\infty
 $$
provided that $S_0^{1-q}-\frac{(1-q)\beta I^p(0)}{p[\mu-\beta S^{q}(0)I^{p-1}(0)]}>0$.
\end{proof}

\subsection{An SIS ODE model}

In this subsection, we provide the results for the corresponding autonomous ODE model
of \eqref{density} with $\mu=0$. That is, consider the following SIS epidemic model:
 \begin{equation}
 \left\{ \begin{array}{llll}
 
{   S'}=-\beta S^qI^p+\gamma I,&\,t>0, \medskip\\
 
{   I'}=\beta S^qI^p-\gamma I,&\,t>0, \medskip\\
 \dis S(0)=S_0>0,\,I(0)= I_0>0.
 \end{array}\right.
 \label{sis}
 \end{equation}
Adding up the first two equations of \eqref{sis}, we find that the total population is a constant, i.e.,
\begin{equation}\label{sis2}
N:=S+I=S_0+I_0, \ \ t\ge 0.
\end{equation}

\begin{Proposition}\label{ODE-SIS}
Let $(S, I)$ be the solution of \eqref{sis}-\eqref{sis2}. The following results hold.
\begin{itemize}
\item Suppose that $p>1$.
\begin{itemize}
\item[\rm{(i)}] If $\gamma<\beta N^*$, where
$$
N^*:= \frac{q^q(p-1)^{p-1}}{(p-1+q)^{p-1+q}} N^{p-1+q},
$$
 then there are two positive steady states, denoted by $(S_*, I_*)$ and $(S^*, I^*)$ with $S_*<S^*$. Moreover, if $S_0< S^*$, then $(S(t), I(t))\rightarrow (S_*, I_*)$ as $t\rightarrow \infty$; if $S_0> S^*$, then $(S(t), I(t))\rightarrow (N, 0)$ as $t\rightarrow \infty$, and if $S_0=S^*$, then $(S(t), I(t))\rightarrow (S^*, I^*)$ as $t\rightarrow \infty$.

\item[\rm{(ii)}] If $\gamma=\beta N^*$,  then there exists a unique positive steady state denoted by $(S_*, I_*)$. Moreover, if $S_0\le S_*$, then $(S(t), I(t))\rightarrow (S_*, I_*)$ as $t\rightarrow \infty$; if $S_0> S_*$, then $(S(t), I(t))\rightarrow (N, 0)$ as $t\rightarrow \infty$.

\item[\rm{(iii)}] If $\gamma>\beta N^*$,  then there is no positive steady state, and $(S(t), I(t))\rightarrow (N, 0)$ as $t\rightarrow \infty$.
\end{itemize}
\item Suppose that $p=1$.
\begin{itemize}
\item[\rm{(i)}] If $\gamma<\beta N^q$, then there exists a unique positive steady state denoted by $(S_*, I_*)$, where $S_*=(\frac{\gamma}{\beta})^\frac{1}{q}$. Moreover, $(S(t), I(t))\rightarrow (S_*, I_*)$ as $t\rightarrow \infty$.

\item[\rm{(ii)}] If $\gamma>\beta N^q$,  then there is no positive steady state, and $(S(t), I(t))\rightarrow (N, 0)$ as $t\rightarrow \infty$.
\end{itemize}

\item Suppose that $p<1$. Then  there exists  a unique positive steady state denoted by $(S_*, I_*)$, and $(S(t), I(t))\rightarrow (S_*, I_*)$ as $t\rightarrow \infty$.

\end{itemize}
\end{Proposition}

\begin{proof}
Since $S(t)+I(t)=N$ for all $t\ge0$, it suffices to consider
\begin{equation}
 \left\{
 \begin{array}{llll}
 {    S'}=-\beta S^q (N-S)^p+\gamma (N-S),&\,t>0, \medskip\\
\dis S(0)=S_0\in (0, N).
 \end{array}
 \right.
 \label{re}
 \end{equation}
 A standard phase plane analysis of \eqref{re} yields the desired results, and we omit the details here.
\end{proof}

\subsection{Some definitions and abstract results on dynamical systems}
We collect the definitions and results  on dynamical systems used in the current paper. These results can be found in  \cite{magal2005global, Zhaobook}, however small modifications are needed since the map is not defined in a complete set in our applications.

Let $(X, d)$ be a complete {\color{red} metric} space, and let $\rho: X\rightarrow [0, \infty)$ be a continuous function. Define
$$
X_0:=\{x\in X: \rho(x)>0\} \ \ \text{and}\ \ \partial X_0:=\{x\in X: \rho(x)=0\}.
$$
For the maps defined on $X$, we adopt all the definitions and terminology in \cite{magal2005global, Zhaobook}.  For \eqref{density}, when $0<p<1$ or $0<q<1$, the solution may fail to be unique if the initial data are not strictly positive. Taking this into consideration, we consider a continuous map $T: X_0\rightarrow X_0$.

For any two sets $A, B\subset X$ and $x\in A$, we let
$$
d(x, A)=\inf_{y\in A}d(x, y)\ \ \text{ and } \ \ \delta(B, A)=\sup_{x\in B} d(x, A).
$$
We say that $A\subset X$ \emph{attracts} $B\subset X_0$ for $T$ if $\lim_{n\rightarrow\infty}\delta(T^n(B), A)=0$.
 \begin{definition}
A continuous map $T: X_0\rightarrow X_0$ is said to be \emph{compact} if for any bounded set $B\subset X$, $T(B\cap X_0)$ is precompact  in $X$; $T$ is \emph{point dissipative} if there is a bounded set $B\subset X$ such that $B$ attracts each point in $X_0$; T is \emph{asymptotically smooth} if for any closed bounded set $B\subset X$ with $T(B\cap X_0)\subset B\cap X_0$, there is a compact set $J\subset B$ such that $J$ attracts $B\cap X_0$.
\end{definition}

The following result is a variant of \cite[Theorem 2.6 (a)]{magal2005global}.
\begin{Proposition}\label{thm_comp}
Let $T: X_0\rightarrow X_0$ be a continuous map. Suppose that $T$ is point dissipative and asymptotically smooth. Then there is a compact set $M\subset X$, which attracts each point in $X_0$ for $T$.
\end{Proposition}

\begin{proof} Since $T$ is point dissipative, there exists a bounded set $B\subset X$ such that for any $x\in X_0$, there exists $N=N(x)$, $T^n(x)\in B\cap X_0$ for all $n\ge N$. Let $J(B)$ be defined by
$$
J(B)=\{y\in B\cap X_0: T^n(y)\in B\cap X_0, \ \text{for all} \ n\ge 0\}.
$$
Clearly, $J(B)$ is not empty. Indeed, for any $x\in X_0$, we have $T^n(x)\in J(B)$ for all $n\ge N$.
Since $T$ is continuous, $\overline{J(B)}\cap X_0=J(B)$. To see this, we first note $J(B)\subset \overline{J(B)}\cap X_0$ as $J(B)\subset X_0$. For any $x\in  \overline{J(B)}\cap X_0$, there exists $\{x_k\}\subset J(B)$ such that $x_k\rightarrow x$. For each $k\ge 0$, since $x_k\in J(B)$, $T^n(x_k)\in B\cap X_0$ for all $n\ge 0$. By the continuity of $T$, $T^n(x)\in B\cap X_0$ for all $n\ge 0$.  Therefore, $x\in J(B)$ and $\overline{J(B)}\cap X_0\subset J(B)$.

Since $\overline{J(B)}$ is bounded with $T(\overline{J(B)}\cap X_0)=T(J(B))\subset J(B)$ and $T$ is asymptotically smooth, there exists a compact set $M\subset X$ such that $M$ attracts $J(B)$. It is not hard to check that $M$ attracts each point in $X_0$.
\end{proof}

\begin{definition}
Let $T: X_0\rightarrow X_0$ be a continuous map. $T$ is said to be \emph{$\rho$-uniformly persistent} if there exists $\epsilon>0$ such that $\liminf_{n\rightarrow\infty} \rho(T^n(x))\ge \epsilon$ for all $x\in X_0$; $T$ is \emph{weakly $\rho$-uniformly persistent} if there exists $\epsilon>0$ such that $\limsup_{n\rightarrow\infty} \rho(T^n(x))\ge \epsilon$ for all $x\in X_0$.
\end{definition}

The proof of the following result is exactly the same as \cite[Proposition 3.2]{magal2005global}.
\begin{Proposition}\label{thm_persistent}
Let $T: X_0\rightarrow X_0$ be a continuous map. Suppose that there exists a compact set $M\subset X$ which attracts each point in $X_0$ for $T$. Then if $T$ is weakly $\rho$-uniformly  persistent, it is $\rho$-uniformly  persistent.
\end{Proposition}

A bounded subset $B\subset X_0$ is \emph{$\rho$-strongly bounded} if there exists $\epsilon>0$ such that $\inf_{x\in B} \rho(x)>0$.  The following result is borrowed from \cite[Theorem 3.8(a)]{magal2005global} with slight modification.

\begin{Proposition}\label{thm_fix}
Let $X$ and $X_0$ be defined as above. Suppose in addition that $X$ is a closed subset of some Banach space with the metric $d$ induced by the norm and $X_0$ is convex.  Let $T: X_0\rightarrow X_0$ be a continuous map. Suppose  that $T$ is compact, point dissipative, $\rho$-uniformly  persistent, and  it maps $\rho$-strongly bounded subsets of $X_0$ to $\rho$-strongly bounded subsets in $X_0$. Then $T$ has a fixed point in $X_0$.

\end{Proposition}

\begin{proof} Let $d_0$ be the metric on $M_0$ as introduced in \cite{magal2005global}:
$$
d_0(x, y)=\left|\frac{1}{\rho(x)}-\frac{1}{\rho(y)}\right|+d(x, y),  \ \ x, y\in M_0.
$$
Then $(M_0, d_0)$ is a complete metric space \cite[Lemma 3.5]{magal2005global}. Thus, all the terminology for $T: (M_0, d_0)\rightarrow (M_0, d_0)$ (eg. global attractor, dissipativity, asymptotical smoothness) can be adopted in the usual sense. Moreover, for any subset $B\subset X_0$, $B$ is bounded in $(X_0, d_0)$ if and only if  it is $\rho$-strongly bounded in $(X_0, d)$; if $B$ is closed (compact) in $(X_0, d)$, then it is closed (compact) in $(X_0, d_0)$; if $B$ is closed (compact) and bounded in $(X_0, d_0)$, then it is closed (compact) in $(X_0, d)$.

Since $T: (X_0, d)\rightarrow (X_0, d)$ is $\rho$-uniformly persistent and point dissipative,   $T: (X_0, d_0)\rightarrow (X_0, d_0)$ is point dissipative. Since $T: (X_0, d)\rightarrow (X_0, d)$ is compact and $T$ maps $\rho$-strongly bounded subsets of $X_0$ to $\rho$-strongly bounded subsets of $X_0$, then $T: (X_0, d_0)\rightarrow (X_0, d_0)$ is compact. To see this, let $B\subset X_0$ be  bounded in $(X_0, d_0)$, then it is $\rho$-strongly bounded in $(X_0, d)$.  Therefore, $T(B)$  is $\rho$-strongly bounded in $(X_0, d)$ and $\overline{T(B)}^d$ is compact in $(X, d)$. Since $T(B)$ is $\rho$-strongly bounded, we have $\overline{T(B)}^{d_0}=\overline{T(B)}^d\subset X_0$, which is compact in $(X_0, d_0)$. Since $T: (X_0, d_0)\rightarrow (X_0, d_0)$ is compact and point dissipative, it has a global attract $A_0\subset X_0$ (\cite[Theorem 1.1.3]{Zhaobook}) and a fixed point $x_0\in A_0$ (\cite[Theorem 1.3.8]{Zhaobook}).
\end{proof}

\end{document}